\documentclass[11pt]{amsart}

\usepackage[dvipsnames]{xcolor}
\usepackage{enumitem}
\usepackage{amssymb,graphicx,mathtools,microtype}
\usepackage{nicefrac}
\usepackage{tikz}
\usetikzlibrary{arrows.meta}
\usepackage[margin=1in]{geometry}
\usepackage{etoolbox}
\usepackage[square,longnamesfirst]{natbib}
\usepackage{xurl}

\AtBeginEnvironment{thebibliography}{\raggedbottom}

\theoremstyle{plain}
\newtheorem{theorem}{Theorem}[section]
\newtheorem{proposition}[theorem]{Proposition}
\newtheorem{lemma}[theorem]{Lemma}
\theoremstyle{definition}

\newcommand{\Z}{\mathbb Z}
\newcommand{\R}{\mathbb R}
\renewcommand{\P}{\mathbb P}
\newcommand{\E}{\mathbb E}

\newcommand{\drift}{\kappa}
\newcommand{\gauge}{\Psi}
\newcommand{\ind}{\mathbf 1}
\newcommand{\dd}{\,\mathrm d}
\newcommand{\defn}[1]{\textbf{#1}}
\newcommand{\nf}{\nicefrac}
\newcommand{\Rin}{R_{\mathrm{in}}}
\newcommand{\Rout}{R_{\mathrm{out}}}
\newcommand{\Rmom}{R_{\mathrm{mom}}}
\numberwithin{equation}{section}

\makeatletter
\patchcmd{\@settitle}{\uppercasenonmath\@title}{\Large}{}{}
\patchcmd{\@setauthors}{\MakeUppercase}{\large}{}{}
\makeatother

\usepackage[citecolor=PineGreen,colorlinks=true,
linkcolor=RoyalBlue,urlcolor=BrickRed]{hyperref}
\hypersetup{pdftitle={Limit shapes for centrally excited walks},
pdfauthor={Ahmed Bou-Rabee, Tuan-Minh Nguyen, Yuval Peres, Arvind Singh}}

\title[Limit shapes for centrally excited walks]{Limit shapes for centrally excited walks}
\author{Ahmed Bou-Rabee}
\address{Department of Mathematics, University of Pennsylvania,
Philadelphia, PA 19104, USA}
\email{ahmedmb@sas.upenn.edu}
\author{Tuan-Minh Nguyen}
\address{School of Mathematics, Monash University,
9 Rainforest Walk, Clayton 3800, Victoria, Australia}
\email{tuanminh.nguyen@monash.edu}
\author{Yuval Peres}
\address{Beijing Institute of Mathematical Sciences and Applications,
Beijing 101408, China}
\email{yperes@bimsa.cn}
\author{Arvind Singh}
\address{CNRS and Laboratoire de Math\'ematiques d'Orsay, Universit\'e Paris-Saclay,
91405 Orsay, France}
\email{arvind.singh@universite-paris-saclay.fr}
\date{}

\begin{document}

\begin{abstract}
A centrally excited random walk on $\Z^d$ moves like simple random walk, except that its first step from each site other than the origin has a drift of fixed size toward it. Kozma (2007) conjectured that after $n$~steps the visited set approximates a ball with radius of order~$n^{\nf{1}{(d+1)}}$. We prove this in every dimension $d\geq2$ and show that the rescaled local times converge uniformly to a cone, which is the potential generated by the drift on a ball. The same argument shows that, when the drift is opposite to a subgradient of a norm, the visited set approximates the ball of that norm. In the plane, the outer and inner radii of the visited set differ by at most $n^{\nf16}$ times a power of~$\log n$, and the exponent~$\nf16$ cannot be lowered. Earlier versions of this work appeared in three independent drafts by \citet{Nguyen2026},
\citet{BouRabeePeres2026CERW} and \citet{Singh2026}. The present paper brings together the results and ideas from these drafts.
\end{abstract}

\maketitle

\begin{figure}[!b]
\centering
\includegraphics[width=0.5\textwidth]{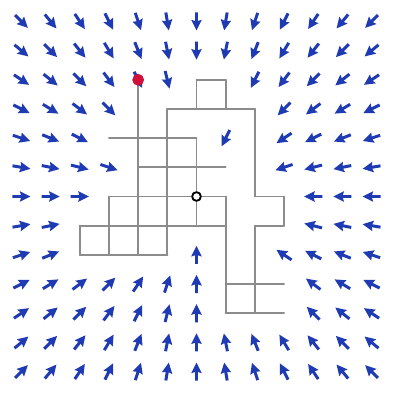}
\caption{The walk after $120$~steps, for $d=2$ and $\drift=0.4$: its path (gray), the origin (open circle) and its position (red dot). At each site~$x\neq0$ that it has not yet left, an arrow shows the mean $-\drift x/|x|$ of its first step from~$x$.}
\label{fig:mechanism}
\end{figure}

\section{Introduction}\label{sec:introduction}

Let $d\geq2$ and $0<\drift<\nf1d$. A \defn{centrally excited random walk} $X_0=0,X_1,X_2,\ldots$ on~$\Z^d$ moves like simple random walk, except that on its first departure from each site~$x\neq0$ it moves to $x\pm e_i$ with probability~$1/(2d)\mp\drift x_i/(2|x|)$ for $1\leq i\leq d$, where $e_1,\ldots,e_d$ is the standard basis of~$\R^d$. Given the path up to that departure, this step has mean~$-\drift x/|x|$, a drift of size~$\drift$ toward the origin (Figure~\ref{fig:mechanism}).

At an Oberwolfach problem session, \citet{Oberwolfach2007} conjectured that the set visited in the first~$n$ steps approximates a ball of radius of order~$n^{\nf{1}{(d+1)}}$. We prove this for every $d\geq2$, identify the radius and the uniform limit of the local times, and bound the fluctuations of the radii and of the local times about these limits. In the plane the inner and outer radii of the visited set fluctuate on the scale~$n^{\nf16}$ up to logarithmic factors; in higher dimensions their fluctuations are bounded by powers of~$\log n$. More generally, when the first-departure drift is opposite to a subgradient of a norm, the visited set approximates the ball of that norm.

\subsection{The limit shape}\label{sec:shape}

For an integer~$n\geq0$ let
\begin{equation*}
\ell_n(x)\coloneqq\sum_{j=0}^{n-1}\ind_{\{X_j=x\}}
\qquad\text{and}\qquad
A_n\coloneqq\{x:\ell_n(x)>0\}\, .
\end{equation*}
The \defn{local time}~$\ell_n(x)$ is the number of departures from~$x$ before time~$n$, and $A_n$ is the \defn{range}; the set of sites visited by time~$n$ is~$A_{n+1}$. Write $B(y,\rho)$ for the open Euclidean ball of radius~$\rho$ about~$y$ and $\omega_d$ for the volume of~$B(0,1)$, and let
\begin{equation}\label{eq:radius}
r_n\coloneqq\left(\frac{(d+1)n}{2d\drift\omega_d}\right)^{\frac{1}{d+1}}\, ,
\end{equation}
so that the cone $2d\drift(r_n-|v|)_+$ has integral~$n$. 

\begin{theorem}[Limit shape]\label{thm:shape}
Let $d\geq2$ and $0<\drift< \frac1d$. Then, almost surely, the centrally excited random walk has the following properties.
\begin{enumerate}[label=\textup{(\roman*)}]
\item \underline{\emph{Shape}}: for every $0<\eta<1$ and all sufficiently large~$n$,
\begin{equation*}
\{x\in\Z^d:|x|<(1-\eta)r_n\}\subset A_n\subset\{x\in\Z^d:|x|<(1+\eta)r_n\}\, .
\end{equation*}
\item \underline{\emph{Local times}}: as $n\to\infty$,
\begin{equation*}
\frac1{r_n}\max_{x\in\Z^d}\bigl|\ell_n(x)-2d\drift(r_n-|x|)_+\bigr|\longrightarrow0\, .
\end{equation*}
\item \underline{\emph{Recurrence}}: the walk visits every site of~$\Z^d$ infinitely often.
\end{enumerate}
\end{theorem}

\subsection{Fluctuations}\label{sec:fluct-intro}

For $x\in\Z^d$ let $C_x\coloneqq x+[-\nf12,\nf12)^d$ be the unit cell centered at~$x$, and let $D_n\coloneqq\bigcup_{x\in A_n}C_x$. Define the \defn{inner radius} and the \defn{outer radius} as
\begin{equation}\label{eq:radii}
\Rin(n)\coloneqq\inf_{y\notin D_n}|y|
\qquad\text{and}\qquad
\Rout(n)\coloneqq\max_{0\leq j\leq n}|X_j|\, .
\end{equation}

\begin{theorem}[Fluctuation bounds]\label{thm:fluctuations}
Let $d\geq2$ and $0<\drift<1/d$. For every $p>0$ there exists $C(d,\drift,p)<\infty$ such that, for every integer~$n\geq2$, the following bounds hold simultaneously with probability at least $1-Cn^{-p}$.
\begin{enumerate}[label=\textup{(\roman*)}]
\item \underline{\emph{Radii}}:
\begin{equation*}
\begin{aligned}
|\Rin(n)-r_n|&\leq C\sqrt{r_n\log n}&&\text{and}\quad&\Rout(n)-r_n&\leq C\sqrt{r_n}(\log n)^{\nf52}&&\text{if } d=2\, ,\\
|\Rin(n)-r_n|&\leq C\log n&&\text{and}\quad&\Rout(n)-r_n&\leq C(\log n)^{d+1}&&\text{if } d\geq3\, .
\end{aligned}
\end{equation*}
\item \underline{\emph{Volume}}:
\begin{equation*}
\bigl|r_n^{-1}D_n\mathbin{\triangle}B(0,1)\bigr|
\leq C\begin{dcases}\sqrt{\frac{\log n}{r_n}},&d=2,\\\frac{\log n}{r_n},&d\geq3\end{dcases}\, .
\end{equation*}
\item \underline{\emph{Local times}}:
\begin{equation*}
\max_{x\in\Z^d}\bigl|\ell_n(x)-2d\drift(r_n-|x|)_+\bigr|
\leq C\begin{cases}\sqrt{r_n}(\log n)^{\nf32},&d=2,\\\sqrt{r_n\log n},&d\geq3\end{cases}\, .
\end{equation*}
\end{enumerate}
In particular, these bounds hold for all sufficiently large~$n$ almost surely. 
\end{theorem}

In the plane, $r_n$ has order~$n^{\nf13}$, so part~(i) bounds the deviations of both radii from~$r_n$ by $n^{\nf16}$ times a power of~$\log n$. The next theorem shows that the planar exponent~$\nf16$ cannot be lowered and that the local-time bound in part~(iii) has the correct order in $d\geq3$.

\begin{theorem}[Sharpness]\label{thm:sharp}
Let $d\geq2$ and $0<\drift<\nf{1}{d}$, and consider the centrally excited random walk.
\begin{enumerate}[label=\textup{(\roman*)}]
\item \underline{\emph{Radii in the plane}}: let $d=2$. For every $p>0$ there exists $c(\drift,p)>0$ such that, for all sufficiently large~$n$,
\begin{equation}\label{eq:planar-polynomial-lower}
\P\bigl(r_n-\Rin(n)\geq c\sqrt{r_n\log n}\bigr)\geq n^{-p}
\qquad\text{and}\qquad
\P\bigl(\Rout(n)-r_n\geq c\sqrt{r_n\log n}\bigr)\geq n^{-p}\, .
\end{equation}
\item \underline{\emph{Iterated logarithm for the radii in the plane}}: let $d=2$. Almost surely,
\begin{equation}\label{eq:planar-lil-lower}
\limsup_{n\to\infty}\frac{r_n-\Rin(n)}{\sqrt{r_n\log\log n}}\geq\frac1{\sqrt{10\pi\drift}}
\qquad\text{and}\qquad
\limsup_{n\to\infty}\frac{\Rout(n)-r_n}{\sqrt{r_n\log\log n}}\geq\frac1{\sqrt{10\pi\drift}}\, .
\end{equation}
\item \underline{\emph{Local times near the origin}}: there exists $c(d,\drift)>0$ such that, for all sufficiently large~$n$, with probability at least $1-n^{-c}$,
\begin{equation}\label{eq:sharp-bulk}
\max_{\substack{x\in\Z^d\\|x|\leq r_n^{\nf12}}}\bigl|\ell_n(x)-2d\drift(r_n-|x|)\bigr|
\geq c\begin{cases}\sqrt{r_n}\log n,&d=2,\\\sqrt{r_n\log n},&d\geq3\end{cases}\, .
\end{equation}
\item \underline{\emph{Difference of the radii in the plane}}: let $d=2$.
\begin{enumerate}[label=\textup{(\alph*)}]
\item For every $p>0$ there exists $c(\drift,p)>0$ such that, for all sufficiently large~$n$,
\begin{equation}\label{eq:planar-width-polynomial}
\P\bigl(\Rout(n)-\Rin(n)\geq c\sqrt{r_n\log n}\bigr)\geq n^{-p}\, .
\end{equation}
\item Almost surely,
\begin{equation}\label{eq:planar-width-lil}
\limsup_{n\to\infty}\frac{\Rout(n)-\Rin(n)}{\sqrt{r_n\log\log n}}\geq\sqrt{\frac{\pi}{3\drift}}\, .
\end{equation}
\item For every $a\geq0$,
\begin{equation}\label{eq:planar-width-small}
\limsup_{n\to\infty}\P\bigl(\Rout(n)-\Rin(n)\leq a\sqrt{r_n}\bigr)\leq1-\exp\Bigl(-\frac{3\drift a^2}{\pi}\Bigr)\, .
\end{equation}
\end{enumerate}
\end{enumerate}
\end{theorem}

Part~(i) shows that the planar bound on $|\Rin(n)-r_n|$ in Theorem~\ref{thm:fluctuations}\,(i) is attained, up to the constant, with probability at least~$n^{-p}$. Part~(iv) shows that the planar radii also differ from each other on the scale~$\sqrt{r_n}$. In the plane, Proposition~\ref{prop:bulk-profile} shows that, at the sites~$x$ with $|x|\leq\theta r_n$ for a fixed $\theta<1$, the local times deviate from the cone by at most $C\sqrt{r_n}\log n$, so part~(iii) is sharp there as well.

\subsection{Limit shapes for other norms}\label{sec:norms-intro}
Let $\gauge$ be a norm with unit ball~$B_\gauge$. Replace $x/|x|$ in the first-departure probabilities by a subgradient~$\xi(x)$ of~$\gauge$ at~$x$, chosen in advance at each site, and assume $\drift\max_i\gauge(e_i)<1/d$. Then the range approximates~$\{\gauge<r_n\}$ and the local times approximate~$2d\drift(r_n-\gauge)_+$, with $|B_\gauge|$ replacing~$\omega_d$ in~\eqref{eq:radius} (Theorem~\ref{thm:norm-shape}).

The theorem applies to norms with corners, such as $\|\cdot\|_1$ and~$\|\cdot\|_\infty$, and Proposition~\ref{prop:norm-shape} gives rates that do not depend on the choice of the subgradients. Theorem~\ref{thm:norm-shape} also extends to the Minkowski functional of every compact convex set containing the origin in its interior, under the analogous condition on~$\drift$ (end of Section~\ref{sec:norms}). Figures~\ref{fig:ranges} and~\ref{fig:profiles} show the range and the local times.

\begin{figure}[t]
\centering
\includegraphics[width=\textwidth]{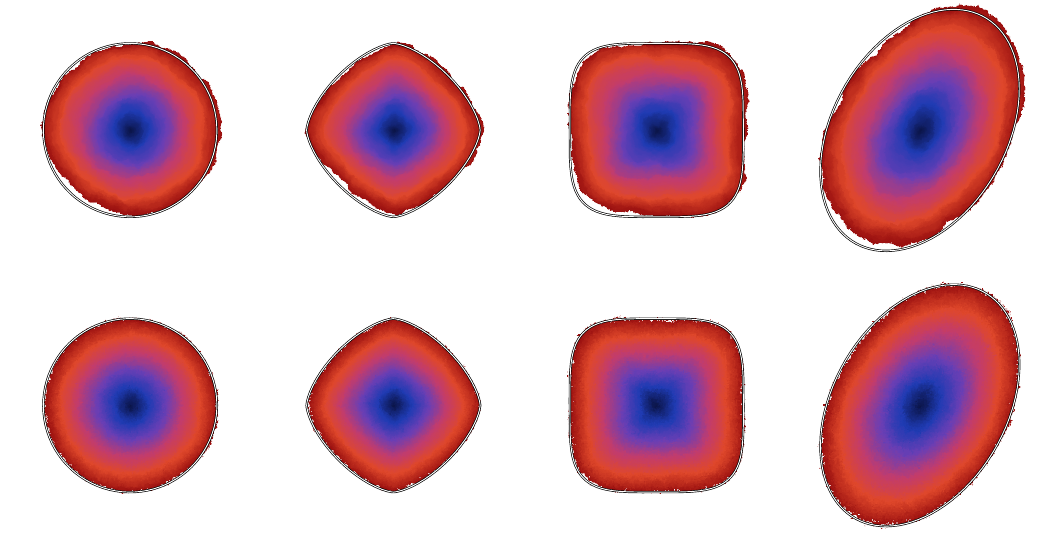}
\caption{Ranges $r_n^{-1}A_n$ at $n=10^9$ for the Euclidean norm, $\|\cdot\|_{3/2}$, $\|\cdot\|_4$ and a norm whose unit ball is an ellipse (left to right), with $d=2$ and $\drift=0.4$ (top) and in the section $x_3=0$ with $d=3$ and $\drift=0.25$ (bottom). The unit balls are outlined, and colors show $(t/n)^{\nf{1}{(d+1)}}$, where $t$ is the time of the first departure from the site.}
\label{fig:ranges}
\end{figure}
\begin{figure}[t]
\centering
\includegraphics[width=\textwidth]{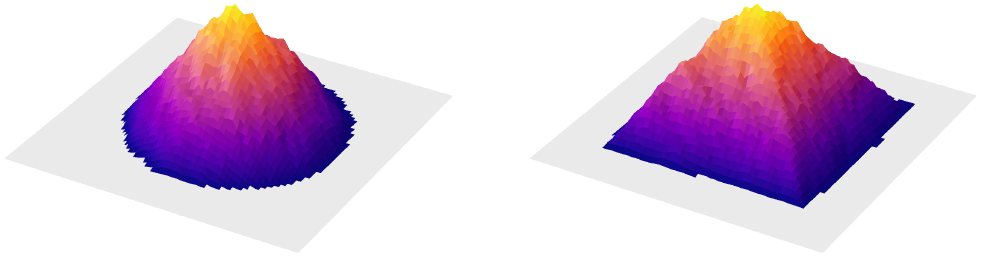}
\caption{Normalized local times $\ell_n(x)/(2d\drift r_n)$ as functions of $x/r_n$, for the Euclidean norm (left) and $\|\cdot\|_\infty$ (right), with $d=2$, $n=10^{10}$ and $\drift=0.4$.}
\label{fig:profiles}
\end{figure}

\subsection{Outline of the proofs}\label{sec:strategy}

The proof compares the local times with an explicit potential. Dynkin's formula for the Green function with pole~$y$ writes~$\ell_n(y)$ as the sum, over the first departures, of the inner product of the drift with the gradient of the Green function, plus a martingale and an error of order at most $\log n$; in the plane, minus the potential kernel takes the place of the Green function. The continuum version of this sum is a potential: for a bounded measurable set~$D\subset\R^d$ and $y\in\R^d$, let
\begin{equation}\label{eq:potential-intro}
U_D(y)\coloneqq\frac{2\drift}{\omega_d}\int_D\frac{v}{|v|}\cdot\frac{v-y}{|v-y|^d}\dd v
\end{equation}
be the potential generated by the vector field~$-\drift v/|v|$ on~$D$. Integration along the rays from~$y$ shows that the potential of a ball is a cone: for every $\rho>0$ and every $y\in\R^d$,
\begin{equation}\label{eq:ballpotential-euclid}
U_{B(0,\rho)}(y)=2d\drift(\rho-|y|)_+\, .
\end{equation}
Replacing $v/|v|$ by~$\nabla\gauge(v)$ gives the cone $2d\drift(\rho-\gauge(y))_+$ for a ball of the norm (Lemma~\ref{lem:ballpotential}).

Near the boundary of the range, the bound on the martingale can be larger than the local times and the potential, so the comparison of~$\ell_n$ with~$U_{D_n}$ alone need not locate the boundary. The \defn{crossing estimate} of Section~\ref{sec:crossing} locates it: with high probability, if the local times outside a ball of the norm are at most~$L$, then up to time~$N$ the walk stays within distance $C(L+1)\log N$ of that ball, because every first-departure drift outside the ball points toward it. We compare the local times with the potential at a contact point: a point~$y_0$ that lies both on the boundary of the largest ball of the norm about the origin contained in~$D_n$ and in the closure of a cell disjoint from~$D_n$. The local time vanishes on that cell, so the potential at~$y_0$ is small, while every point of~$D_n$ outside the ball contributes nonnegatively to it.
\begin{enumerate}[label=\arabic*.,leftmargin=*]
\item \emph{Local times and the potential} (Section~\ref{sec:local}). Once the local times are $O(r_n)$, they differ from~$U_{D_n}$ by at most $C\sqrt{r_n}\log n$ (Lemma~\ref{lem:local}). Freedman's inequality bounds the martingales in Dynkin's formula uniformly over sites. Replacing the sum over first departures by~$U_{D_n}$ changes it by $O(\log n)$, also for norms with corners; convexity bounds the variation of the subgradients across a cell (Lemma~\ref{lem:cell}).
\item \emph{Crossing and coarse bounds} (Section~\ref{sec:coarse}). Let $\tau_R$ be the first time at which the walk leaves the ball of the norm with radius~$R$. Before~$\tau_R$, and while the martingales in Dynkin's formula stay below~$R$ in absolute value, the formula bounds each local time by~$CR$, and hence the elapsed time by~$CR^{d+1}$; Freedman's inequality then shows that the martingales are unlikely to exceed~$R$ before~$\tau_R$. At time~$\tau_R$, the comparison at the contact point and the crossing estimate show that the range contains all sites of a ball of radius~$cR$, and that the local times have order~$R$ on the concentric ball of half this radius; so $\tau_R\geq cR^{d+1}$. Hence $|A_n|$ has order~$r_n^d$, and $\Rout(n)$ and the largest local time have order~$r_n$ (Proposition~\ref{prop:coarse}).
\item \emph{The limit shape} (Section~\ref{sec:norms}). Let $b$ be the radius of the largest ball of the norm about the origin contained in~$D_n$. By convexity, every point~$v$ of~$D_n$ outside this ball contributes at least $c(\gauge(v)-b)r_n^{-d}$ to the potential at the contact point. Dynkin's formula at the center of the adjacent cell bounds this potential by the martingale there, and Dynkin's formula for the square of the Green function or potential kernel bounds the bracket of that martingale in terms of the local time at its pole, which vanishes (Lemma~\ref{lem:contact}). So the part of~$D_n$ outside the ball has small potential, which gives the cone for the local times, and the martingale $|X_n|^2-n+2\drift\sum_{x\in A_n}\gauge(x)$ gives $b\approx r_n$. The local times outside $\{\gauge<r_n\}$ are then small, and the crossing estimate shows that no visited site lies far from $\{\gauge\leq r_n\}$.
\item \emph{The inner radius} (Section~\ref{sec:contact}). For the Euclidean norm, every point~$v$ of~$D_n$ outside~$B(0,\Rin(n))$ contributes at least $c|v-y_0|^{2-d}/r_n$ to the potential at the contact point~$y_0$, a bound that does not vanish near the boundary of the ball. With Lemma~\ref{lem:contact}, this bounds the volume of~$D_n$ outside the ball and sharpens the inner-radius estimate (Proposition~\ref{prop:inner}).
\item \emph{The outer radius} (Section~\ref{sec:outer}). At a point beyond the outer radius the local time vanishes, and the points of~$D_n$ near it contribute negatively to the potential there. So a bound on the positive contributions, which uses the curvature of the sphere of radius~$\Rout(n)$, shows that little of~$D_n$ lies near this sphere. This bounds the local times outside~$B(0,\Rin(n))$; the crossing estimate then gives the outer-radius bound (Proposition~\ref{prop:stronger-outer}), and the maximum principle gives the uniform local-time bound.
\end{enumerate}
The cone determines the limits of the brackets of the martingales in Dynkin's formula and of the martingale $|X_n|^2-n+2\drift\sum_{x\in A_n}|x|$, so the martingale limit theorems give central limit theorems and laws of the iterated logarithm for $\sum_{x\in A_n}|x|$ and for the local times at fixed sites (Section~\ref{sec:sharp}). The lower bounds in Theorem~\ref{thm:sharp}\,(i), (iii) and~(iv)(a) need a martingale to exceed about $(\log n)^{\nf12}$ times its standard deviation, which exponential martingales give (Section~\ref{sec:lower}): the martingale $|X_n|^2-n+2\drift\sum_{x\in A_n}|x|$ gives the lower bounds on the planar radii, and many weakly correlated martingales in Dynkin's formula give the lower bound on the largest deviation of the local times near the origin. The lower bounds on the difference of the radii come from the vector martingale $X_n+\drift\sum_{x\in A_n\setminus\{0\}}x/|x|$: the vectors $x/|x|$ cancel inside~$B(0,\Rin(n))$, so fluctuations of order~$\sqrt n$ of this martingale force fluctuations of order~$\sqrt{r_n}$ of the difference of the radii. The two-dimensional Gaussian limit of this martingale gives the bound in Theorem~\ref{thm:sharp}\,(iv)(c). Figure~\ref{fig:dependence} records the dependencies.

\begin{figure}[t]
\centering
\begin{tikzpicture}[x=0.96cm,y=1cm,
  box/.style={draw,rounded corners=2pt,align=center,font=\small,
    inner sep=3pt,minimum height=0.9cm},
  arr/.style={-{Stealth[length=5pt]},thin}]
\node[box] (L) at (0,0) {Lemma~\ref{lem:local}\\ local times and\\ the potential};
\node[box] (P3) at (3,0) {Proposition~\ref{prop:coarse}\\ coarse bounds};
\node[box] (P4) at (6,0) {Proposition~\ref{prop:norm-shape}\\ limit shape\\ with rates};
\node[box] (TN) at (9,0) {Theorem~\ref{thm:norm-shape}\\ limit shape\\ for a norm};
\node[box] (TA) at (12,0) {Theorem~\ref{thm:shape}};
\node[box] (P5) at (3,-2.1) {Proposition~\ref{prop:inner}\\ inner radius};
\node[box] (P6) at (6,-2.1) {Proposition~\ref{prop:stronger-outer}\\ outer radius};
\node[box] (TB) at (9,-2.1) {Theorem~\ref{thm:fluctuations}};
\node[box] (S8) at (12,-2.1) {Theorems~\ref{thm:moment-fluctuations}\\ and~\ref{thm:site-fluctuations}\\ limit laws};
\node[box] (S9) at (8,-4.2) {Lemmas~\ref{lem:exp-deviation} and~\ref{lem:separated-brackets},\\ Proposition~\ref{prop:bulk-profile}\\ deviation bounds};
\node[box] (TC) at (12,-4.2) {Theorem~\ref{thm:sharp}};
\draw[arr] (L) -- (P3);
\draw[arr] (P3) -- (P4);
\draw[arr] (P4) -- (TN);
\draw[arr] (TN) -- (TA);
\draw[arr] (P3) -- (P5);
\draw[arr] (P5) -- (P6);
\draw[arr] (P6) -- (TB);
\draw[arr] (P5) to[bend right=18] (TB);
\draw[arr,dashed] (P4) -- (P5);
\draw[arr] (P4) -- (P6);
\draw[arr] (TA) -- (S8);
\draw[arr] (TB) -- (S8);
\draw[arr] (TB) -- (S9);
\draw[arr] (S8) -- (TC);
\draw[arr] (S9) -- (TC);
\draw[arr] (TA.east) to[out=0,in=0,looseness=1.4] (TC.east);
\end{tikzpicture}
\caption{Dependencies among the main results. A solid arrow from one result to another means that the first is used in the proof of the second; a dashed arrow means that estimates from Section~\ref{sec:norms} are used.}
\label{fig:dependence}
\end{figure}
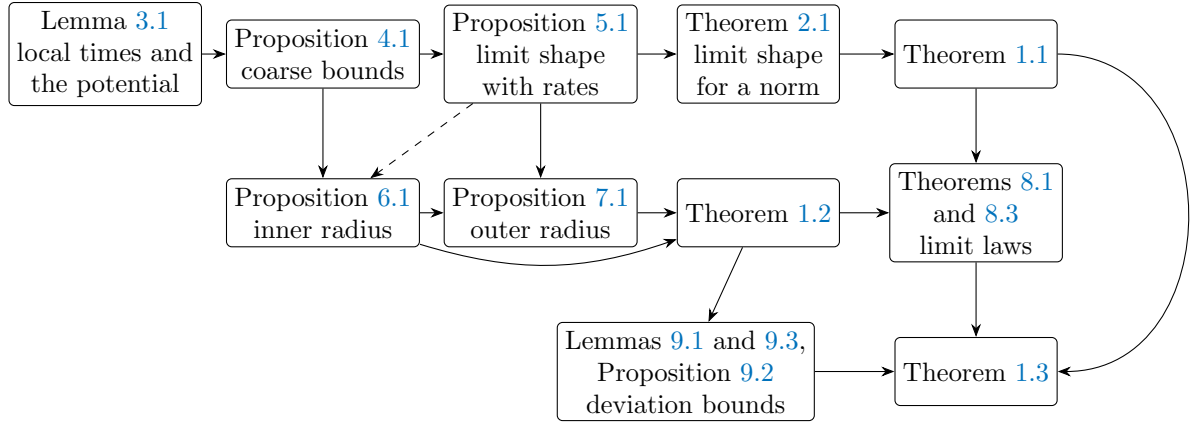

The local times play the role of the odometer of the divisible sandpile, which \citet{LevinePeres2009} compared with an explicit potential. Martingales built from differences of the potential kernel in the plane and from hitting probabilities in higher dimensions, and separate estimates that rule out thin tentacles, as the crossing estimate does here, also appear in the bounds of~\citet*{JerisonLevineSheffield2012,JerisonLevineSheffield2013} for internal diffusion-limited aggregation. Proposition~\ref{prop:log-lower}, which bounds $\max\{r_n-\Rin(n),\Rout(n)-r_n\}$ below by $c\log n$, adapts the argument of \citet{AsselahGaudilliere2014}. For internal diffusion-limited aggregation in $d\geq3$, they proved that for some $\alpha>0$, conditionally on the inner radius of the cluster of $|B(0,r)\cap\Z^d|$~particles exceeding $r-\alpha\sqrt{\log r}$, the probability that the outer radius of the cluster of $|B(0,s)\cap\Z^d|$~particles exceeds $s+\alpha\sqrt{\log s}$ for some $s>r$ tends to one as $r\to\infty$.

\subsection{Related work}\label{sec:related}

\citet{Kozma2013Recurrence} gave a proof, dating from 2006, that the walk excited to the center is recurrent in every dimension; the note explains its main lemma by examples. Theorem~\ref{thm:shape}\,(iii) proves this recurrence for the walk defined above. \citet{Kozma2013} reported the recurrence, and \citet*{DemboHuangSidoravicius2014,DemboHuangSidoravicius2014Growing} cited the proof as an unpublished manuscript titled ``Centrally excited random walk is recurrent''; we take the name of the walk from that title. \citet*{BenjaminiDuminilCopinKozmaLucas2020} also recorded the predicted radius~$n^{\nf{1}{(d+1)}}$.

For the excited random walk with first-visit drift in a fixed direction, \citet*{BenjaminiWilson2003} proved transience in $d\geq2$ and positive lower speed in $d\geq4$. \citet{Kozma2003Speed,Kozma2005Speed} obtained positive lower speed in $d=3$ and $d=2$, and \citet{BerardRamirez2007} proved a law of large numbers and a central limit theorem for the coordinate in the drift direction, in every $d\geq2$. \citet*{vanderHofstadHolmes2010} proved that in $d\geq9$ the speed is strictly increasing in the size of the drift. \citet*{MenshikovPopovRamirezVachkovskaia2012} proved positive lower speed in every $d\geq2$ for uniformly elliptic walks with bounded jumps, a uniformly positive drift in a fixed direction at first visits and zero drift at later visits. \citet*{AmirBenjaminiKozma2008} proved recurrence for a walk in the upper half of~$\Z^3$ that steps down at new sites. The balanced excited random walk on~$\Z^2$ steps vertically on its first departure from each site and horizontally on later departures. \citet*{AngelHolmesRamirez2023} proved that, with probability tending to one, it visits at least $n^{\nf47}/(\log n)^2$ and at most $na_n/\sqrt{\log\log n}$ sites in $n$~steps, for every sequence $a_n\to\infty$; so the diffusively rescaled walk converges to a Brownian motion on the horizontal axis.

\citet{Zerner2005,Zerner2006} gave criteria for recurrence and transience on~$\Z$ and on strips for multi-excited, or cookie, walks, in which the $i$th departure from a site follows a law prescribed by the $i$th cookie at that site. For the walk on~$\Z$ whose $i$th departure from every site is to the right with probability~$p_i\in[\nf12,1)$ for $i\leq M$ and~$\nf12$ for $i>M$, \citet{Zerner2005} proved transience if and only if $\sum_{i=1}^M(2p_i-1)>1$, and \citet*{BasdevantSingh2008} proved that the speed is positive if and only if this sum exceeds~$2$. \citet*{KosyginaMountfordPeterson2022} proved that walks on~$\Z$ whose cookie probabilities are a fixed $(0,1)$-valued function of independent copies, one for each site, of a finite-state Markov chain with a unique stationary distribution and stationary mean cookie probability~$\nf12$ converge, in the recurrent regime excluding its boundary, under diffusive scaling to a multiple of Brownian motion perturbed at its extrema. Excited random walks also model foraging in a medium that the walker depletes: \citet*{CamposCristinMendez2021} studied a continuum model in one dimension in which the walker depletes an interval at constant speed until it misses the food, which happens at a fixed rate. When the walker, after missing food, first moves diffusively and later ballistically, they computed the mean time to leave the depleted interval and the asymptotic growth of that interval.

On a tree in which every vertex other than the root has at least three neighbors, \citet{Volkov2003} proved that the walk that steps toward the root with a fixed probability at each first visit, and moves like simple random walk at later visits, is transient for every value of that probability. \citet*{BasdevantSingh2009} gave a recurrence and transience criterion on regular trees for cookie walks whose cookies may push toward or away from the root, and showed that, unlike on~$\Z$, whether the walk is recurrent can depend on the order of the cookies at a vertex. \citet*{LeNguyen2026} proved that the once-excited walk on a tree whose first step from a vertex~$v$ goes to the parent with $1+\alpha_v\deg(v)$ times the probability of going to each child, with i.i.d.\ nonnegative~$\alpha_v$, is, under the annealed law, transient or recurrent according as the branching-ruin number of the tree is above or below $2-\E[(1+\alpha_v)^{-1}]$. See~\citet{KosyginaZerner2013} for a survey.

The radius~$n^{\nf{1}{(d+1)}}$ and range size~$n^{\nf{d}{(d+1)}}$ are also predicted for once-reinforced walks with strong reinforcement~\citep{Beffara2011,ElboimKozma2026}, and supported by simulations~\citep{FosterGrassbergerPaczuski2009}. \citet*{OrdemannBerkolaikoHavlinBunde2000,OrdemannEtAl2001} found, by scaling arguments and simulations, the exponents $\nf{1}{(d+1)}$ for the displacement and $\nf{d}{(d+1)}$ for the number of visited sites of self-attracting walks above a critical attraction. For once-reinforced random walk on~$\Z^d$ with $d\geq2$, \citet{BouRabeePeres2026ORRW} proved the lower bound~$cn^{\nf{d}{(d+1)}}$ on the expected range for every reinforcement parameter. The matching upper bound and limit shape remain conjectural. See~\citet{Pemantle2007} for reinforced processes.

\citet{DemboGroismanHuangSidoravicius2021} proved an averaging principle and shape theorem for growth models motivated by centrally excited and once-reinforced walks. Their simulations with inward first steps along a coordinate of largest absolute value or along all coordinates produced squares and diamonds, respectively. Theorem~\ref{thm:norm-shape} proves these shapes for first-departure drifts opposite to subgradients of~$\|\cdot\|_\infty$ and~$\|\cdot\|_1$. For related continuum growth models, \citet*{DemboYang2025Flow,DemboYang2025KPZ} obtained a scaling limit of the interface given by a flow-type equation, and fluctuations given by a regularized Kardar--Parisi--Zhang-type equation.

Internal diffusion-limited aggregation~\citep{LawlerBramsonGriffeath1992} and rotor-router aggregation~\citep{LevinePeres2009} have Euclidean ball limit shapes. For clusters of radius~$r$, internal diffusion-limited aggregation has radius fluctuations of order at most $\log r$ in the plane and $(\log r)^{\nf12}$ in higher dimensions~\citep{JerisonLevineSheffield2012,JerisonLevineSheffield2013,AsselahGaudilliere2013,AsselahGaudilliere2013Polynomial}, and its rescaled fluctuations approach a variant of the Gaussian free field~\citep{JerisonLevineSheffield2014}. In contrast, for the centrally excited walk in the plane the radii and their difference fluctuate on the scale~$\sqrt{r_n}$ up to logarithmic factors (Theorems~\ref{thm:fluctuations} and~\ref{thm:sharp}). See~\citet*{LevinePeres2017} for aggregation models.

For Eulerian walkers, also called rotor walks~\citep{PriezzhevDharDharKrishnamurthy1996}, simulations suggested a circular planar range of radius of order~$n^{\nf13}$ after $n$~steps~\citep{KapriDhar2009}. \citet*{FlorescuLevinePeres2016} proved that every simple rotor walk on~$\Z^d$ visits at least $cn^{\nf{d}{(d+1)}}$ sites in $n$~steps. For clockwise planar rotors with independent uniform initial states, \citet*{BouRabeePeres2026} proved convergence of the rescaled range to a deterministic compact convex set; whether it is a disk remains open.

\subsection{Notation and conventions}\label{sec:notation}

\begin{itemize}
\item The Euclidean norm is $|x|$, and $\|x\|_q$ is the $\ell^q$~norm for $1\leq q\leq\infty$. The cardinality of a set~$A\subset\Z^d$ is~$|A|$, and the Lebesgue measure of a set~$D\subset\R^d$ is~$|D|$. The symmetric difference is~$\triangle$, and $t_+\coloneqq\max\{t,0\}$.
\item Let $u_x\coloneqq x/|x|$ for $x\in\R^d\setminus\{0\}$ and $u_0\coloneqq0$, and let $\xi(0)\coloneqq0$.
\item The letters $c>0$ and $C>0$ denote constants depending only on $d$, $\drift$ and the exponent~$p$ of the failure probability under discussion, and in Sections~\ref{sec:norm-setup}--\ref{sec:norms} also on the norm~$\gauge$, but not on the choice of the subgradients~$\xi(x)$; they may change from line to line. Every further dependence is stated, and a subscript on~$O$, as in~$O_y$, means that the implied constant also depends on~$y$.
\item The $\sigma$-field generated by $X_0,\ldots,X_j$ is~$\mathcal F_j$. For martingales $Y$ and~$Z$, the brackets $\langle Y\rangle_n$ and~$\langle Y,Z\rangle_n$ are the predictable quadratic variation and covariation through time~$n$.
\item The indicator that the step at time~$j$ is a first departure is $I_j\coloneqq\ind_{\{X_j\notin A_j\}}$, so that $\E(X_{j+1}-X_j\mid\mathcal F_j)=-\drift I_j\xi(X_j)$, where $\xi(x)=u_x$ for the Euclidean norm.
\item The function~$\widetilde\ell_n$ on~$\R^d$ equals $\ell_n(x)$ on~$C_x$ for every $x\in\Z^d$; it vanishes off~$D_n$.
\item All estimates concern integers~$n\geq2$, so that $\log n>0$. In a statement that holds with probability at least $1-Cn^{-p}$ for all sufficiently large~$n$, enlarging~$C$ makes it hold for every $n\geq2$.
\end{itemize}

\subsection{Lean formalization}\label{ssec:formalization}
The theorems, lemmas and propositions of this paper were formalized in Lean~4, building on Mathlib. The formal statements are those of this paper, and the formal proofs follow its arguments. No formal proof uses \texttt{sorry} or an axiom beyond the three standard axioms \texttt{propext}, \texttt{Classical.choice} and \texttt{Quot.sound}. The development is at \url{https://github.com/nitromannitol/Centrally-Excited-Random-Walk}.

\subsection*{Use of artificial intelligence}
The authors developed the main ideas of this work and directed its mathematical development. Throughout the project, the authors used large language models as research assistants: to explore specific questions arising in the proofs, to check the arguments and the cited results, and to audit drafts of the paper. The exposition, the simulations and the figures were produced with AI agents, primarily Claude (Anthropic), under the close supervision of the authors. The authors take responsibility for the mathematical content and presentation of the paper.

\subsection*{Acknowledgments}
The research of Y. Peres was supported by the National Natural Science Foundation of China grant RFIS-W2531011.

\subsection*{Earlier drafts}
Earlier versions of this work appeared in three independent drafts by \citet{Nguyen2026},
\citet{BouRabeePeres2026CERW} and \citet{Singh2026}. The present paper brings together the results and ideas from these drafts.

\section{Drifts given by a norm}\label{sec:norm-setup}

In this section we define the walk whose first-departure drift is opposite to a subgradient of a norm and state its limit shape theorem, Theorem~\ref{thm:norm-shape}, of which Theorem~\ref{thm:shape} is the Euclidean case. We also prove that the potential of a ball of the norm is a cone (Lemma~\ref{lem:ballpotential}).

Let $\gauge$ be a norm on~$\R^d$ and let $B_\gauge$ be its unit ball. A \defn{subgradient} of~$\gauge$ at~$x\in\R^d$ is a vector $\xi\in\R^d$ such that $\gauge(y)\geq\gauge(x)+\xi\cdot(y-x)$ for every $y\in\R^d$, and $\partial\gauge(x)$ is the set of subgradients at~$x$; if $\gauge$ is differentiable at~$x$, then $\partial\gauge(x)=\{\nabla\gauge(x)\}$. At each site~$x\neq0$ fix $\xi(x)\in\partial\gauge(x)$, with coordinates~$\xi_i(x)$. The \defn{centrally excited random walk with norm~$\gauge$} moves like simple random walk, except that on its first departure from each site~$x\neq0$ it moves to $x\pm e_i$ with probability
\begin{equation}\label{eq:kernel}
\frac1{2d}\mp\frac\drift2\xi_i(x)
\qquad\text{for } 1\leq i\leq d\, .
\end{equation}
We assume that
\begin{equation}\label{eq:ellipticity}
\drift\max_{1\leq i\leq d}\gauge(e_i)<\frac1d\, .
\end{equation}
Every $\xi\in\partial\gauge(x)$ satisfies Euler's relation $\xi\cdot x=\gauge(x)$, so the definition of a subgradient gives
\begin{equation}\label{eq:convexity}
\xi\cdot y\leq\gauge(y)\qquad\text{for every }y\in\R^d\, ,
\end{equation}
with equality at $y=x$. Hence $|\xi_i(x)|\leq\gauge(e_i)$, and~\eqref{eq:ellipticity} makes the probabilities~\eqref{eq:kernel} positive. The step at the first departure from~$x$ has conditional mean~$-\drift\xi(x)$, whose inner product with~$x$ is $-\drift\gauge(x)<0$. The Euclidean norm, with $\xi(x)=u_x$, gives the walk of Section~\ref{sec:introduction}, and then~\eqref{eq:ellipticity} reads $\drift<\nf1d$. Let
\begin{equation}\label{eq:radius-norm}
r_n\coloneqq\left(\frac{(d+1)n}{2d\drift|B_\gauge|}\right)^{\nf{1}{(d+1)}}\, ,
\end{equation}
which is~\eqref{eq:radius} when $\gauge$ is the Euclidean norm; the cone $2d\drift(r_n-\gauge(v))_+$ has integral~$n$ over~$\R^d$.

In Sections~\ref{sec:norm-setup}--\ref{sec:norms}, $\gauge$ is a norm, $\drift>0$ satisfies~\eqref{eq:ellipticity}, $r_n$ is given by~\eqref{eq:radius-norm}, and $U_D$ is the potential~\eqref{eq:potential-norm} defined below. Let $\Lambda_\gauge\coloneqq\max_{|u|=1}\gauge(u)$ and $\lambda_\gauge\coloneqq\min_{|u|=1}\gauge(u)$; by~\eqref{eq:convexity}, every subgradient~$\xi$ of~$\gauge$ satisfies $|\xi|\leq\Lambda_\gauge$. In Sections~\ref{sec:contact}--\ref{sec:lower}, $\gauge$ is the Euclidean norm and $\xi(x)=u_x$; then $r_n$ is given by~\eqref{eq:radius}, and $U_D$ is the potential~\eqref{eq:potential-intro}.

\begin{theorem}[Limit shape for a norm]\label{thm:norm-shape}
Let $d\geq2$, let $\gauge$ be a norm on~$\R^d$, let $\xi(x)\in\partial \gauge(x)$ be chosen at the sites $x\in\Z^d\setminus\{0\}$, let $\drift>0$ satisfy~\eqref{eq:ellipticity}, and let $r_n$ be as in~\eqref{eq:radius-norm}. Then, almost surely, the centrally excited random walk with norm~$\gauge$ has the following properties.
\begin{enumerate}[label=\textup{(\roman*)}]
\item \underline{\emph{Shape}}: for every $0<\eta<1$ and all sufficiently large~$n$,
\begin{equation*}
\{x\in\Z^d:\gauge(x)<(1-\eta)r_n\}\subset A_n\subset\{x\in\Z^d:\gauge(x)<(1+\eta)r_n\}\, .
\end{equation*}
\item \underline{\emph{Local times}}: as $n\to\infty$,
\begin{equation*}
\frac1{r_n}\max_{x\in\Z^d}\bigl|\ell_n(x)-2d\drift(r_n-\gauge(x))_+\bigr|\longrightarrow0\, .
\end{equation*}
\item \underline{\emph{Recurrence}}: the walk visits every site of~$\Z^d$ infinitely often.
\end{enumerate}
\end{theorem}

Proposition~\ref{prop:norm-shape} gives rates for parts~(i) and~(ii), and for the Euclidean norm Theorem~\ref{thm:fluctuations} gives smaller ones. Part~(i) gives $|A_n|/r_n^d\to|B_\gauge|$.

\subsection{The potential of a norm ball}\label{sec:ballpotential}

For a norm~$\gauge$, the potential~\eqref{eq:potential-intro} is defined with the gradient of~$\gauge$ in place of~$v/|v|$: for a bounded measurable set~$D\subset\R^d$ and $y\in\R^d$, let
\begin{equation}\label{eq:potential-norm}
U_D(y)\coloneqq\frac{2\drift}{\omega_d}\int_D\nabla\gauge(v)\cdot\frac{v-y}{|v-y|^d}\dd v
\end{equation}
be the potential generated by the vector field~$-\drift\nabla\gauge$ on~$D$; by Rademacher's theorem, $\nabla\gauge$ exists almost everywhere.

For every Lipschitz function~$f$ with compact support and every $y\in\R^d$,
\begin{equation}\label{eq:gradient-identity}
\int_{\R^d}\nabla f(v)\cdot\frac{v-y}{|v-y|^d}\dd v=-d\omega_df(y)\, ,
\end{equation}
because, by Rademacher's theorem, for almost every unit vector~$u$ the integrand times~$\dd v$ is $\partial_tf(y+tu)\dd t\dd S(u)$ in polar coordinates $v=y+tu$ about~$y$. Here $S$ is the surface measure on the unit sphere, of total mass~$d\omega_d$.

\begin{lemma}[The potential of a norm ball]\label{lem:ballpotential}
Let $\gauge$ be a norm, let $\rho>0$, and let $y\in\R^d$. Then
\begin{equation}\label{eq:ballpotential}
U_{\{\gauge<\rho\}}(y)=2d\drift(\rho-\gauge(y))_+\, .
\end{equation}
\end{lemma}

\begin{proof}
Apply~\eqref{eq:gradient-identity} to $f=(\rho-\gauge)_+$, whose gradient is $-\ind_{\{\gauge<\rho\}}\nabla\gauge$ almost everywhere.
\end{proof}

\section{Local times and the potential}\label{sec:local}

In this section we prove Lemma~\ref{lem:local}, which compares the local times with the potential~$U_{D_n}$, as outlined in step~1 of Section~\ref{sec:strategy}.

\begin{lemma}\label{lem:local}
Let $p>0$. There exists $C(d,\drift,\gauge,p)<\infty$ such that, for every integer~$n\geq2$, with probability at least $1-Cn^{-p}$, every $y\in\R^d$ with $|y|\leq2n$ satisfies
\begin{equation}\label{eq:approx}
|\widetilde\ell_n(y)-U_{D_n}(y)|\leq C\log n+C\begin{cases}\bigl(\max_x\ell_n(x)\bigr)^{\nf12}\log n,&d=2,\\\bigl(\max_x\ell_n(x)\log n\bigr)^{\nf12},&d\geq3\end{cases}\, .
\end{equation}
\end{lemma}

\subsection{Dynkin's formula for the local times}\label{sec:dynkin}

Let $\Delta f(x)\coloneqq(2d)^{-1}\sum_{|e|=1}(f(x+e)-f(x))$, summed over the unit vectors $e\in\Z^d$, be the discrete Laplacian. For $d=2$, let $g$ be the potential kernel of simple random walk, normalized by $g(0)=0$. For $d\geq3$, let $g\coloneqq-G$, where the Green function~$G(x)$ is the expected number of visits to~$x$ by simple random walk from the origin. In both cases $\Delta g=\ind_{\{0\}}$; for $d=2$ see \citet*[Proposition~4.4.2]{LawlerLimic2010}. By \citet[Theorems~4.3.1 and~4.4.4]{LawlerLimic2010}, as $|x|\to\infty$,
\begin{equation}\label{eq:kernel-asymptotics}
g(x)=\begin{dcases}
\frac2\pi\log|x|+c_g+O(|x|^{-2}),&d=2,\\
-\frac{2}{(d-2)\omega_d}|x|^{2-d}+O(|x|^{-d}),&d\geq3
\end{dcases}\, ,
\end{equation}
where $c_g$ is a constant. Let $\bar\nabla g(x)$ be the vector with coordinates $(g(x+e_i)-g(x-e_i))/2$ for $1\leq i\leq d$. Then
\begin{equation}\label{eq:gradient}
\bar\nabla g(x)=\frac2{\omega_d}\frac{x}{|x|^d}+O(|x|^{-d})
\quad\text{and}\quad
|g(x+e)-g(x)|\leq C(1+|x|)^{1-d}\, ,
\end{equation}
the first for $x\neq0$, the second for every $x\in\Z^d$ and every unit vector $e\in\Z^d$; see \citet[Corollaries~4.3.3 and~4.4.5]{LawlerLimic2010}.

For $f\colon\Z^d\to\R$ with bounded nearest-neighbor increments, let
\begin{equation*}
\mathcal M^f_t\coloneqq f(X_t)-f(X_0)-\sum_{j<t}\E\bigl(f(X_{j+1})-f(X_j)\bigm|\mathcal F_j\bigr)
\end{equation*}
be the martingale in Dynkin's formula for~$f$. By~\eqref{eq:kernel}, the conditional expectation in the sum is $\Delta f(X_j)-\drift I_j\xi(X_j)\cdot\bar\nabla f(X_j)$, with $\bar\nabla f$ defined like~$\bar\nabla g$. For $y\in\Z^d$ write $g_y\coloneqq g(\cdot-y)$ and $\mathcal M^y\coloneqq\mathcal M^{g_y}$. Since $\Delta g_y=\ind_{\{y\}}$ and each site of~$A_n$ has its first departure before time~$n$,
\begin{equation}\label{eq:dynkin}
\ell_n(y)=g(X_n-y)-g(-y)
+\drift\sum_{x\in A_n}\xi(x)\cdot\bar\nabla g(x-y)-\mathcal M_n^y\, .
\end{equation}
The difference $g(X_n-y)-g(-y)$ is $O(\log n)$ for $d=2$ and $O(1)$ for $d\geq3$, uniformly over $y\in\Z^d$ with $|y|\leq3n$. By~\eqref{eq:gradient}, the increments of~$\mathcal M^y$ are bounded, and a departure from~$x$ adds at most $C(1+|x-y|)^{2-2d}$ to the bracket of~$\mathcal M^y$; so, for every integer $t\geq0$,
\begin{equation}\label{eq:bracket}
\langle\mathcal M^y\rangle_t\leq C\sum_{j<t}(1+|X_j-y|)^{2-2d}\, .
\end{equation}

\begin{lemma}\label{lem:freedman}
Let $p>0$. There exists $C(p)<\infty$ such that the following holds for every integer~$n\geq2$ and every real martingale~$Z$ whose increments have absolute value at most~$1$: with probability at least $1-Cn^{-p}$,
\begin{equation}\label{eq:freedman}
|Z_n-Z_0|\leq C\bigl(\sqrt{\langle Z\rangle_n\log n}+\log n\bigr)\, .
\end{equation}
\end{lemma}

\begin{proof}
For $0\leq i\leq\log_2n+1$, \citet*[Theorem~1.6]{Freedman1975} bounds the probability that $\langle Z\rangle_n\leq2^i$ and $|Z_n-Z_0|>C(p)(\sqrt{2^i\log n}+\log n)$ by $2n^{-p-1}$. Since $\langle Z\rangle_n\leq n$, a union bound over these $O(\log n)$ values of~$i$ proves~\eqref{eq:freedman}.
\end{proof}

We apply Lemma~\ref{lem:freedman} to constant multiples of possibly stopped martingales, with union bounds over polynomially many deterministic sites; the resulting estimates then hold also at sites that depend on the path. In particular, for every $p>0$, with probability at least $1-Cn^{-p}$, simultaneously for all $y\in\Z^d$ with $|y|\leq3n$,
\begin{equation}\label{eq:localmart}
|\mathcal M_n^y|\leq C\bigl(\sqrt{\langle\mathcal M^y\rangle_n\log n}+\log n\bigr)\, .
\end{equation}

To replace $\xi(x)$ by~$\nabla\gauge$ on~$C_x$ in~\eqref{eq:dynkin}, we bound the variation of~$\nabla\gauge$ by the distributional Laplacian~$\mu$ of~$\gauge$. Since $\gauge$ is convex, $\mu$ is a positive Radon measure, and $\int\phi\dd\mu=-\int\nabla\gauge\cdot\nabla\phi$ for every Lipschitz function~$\phi$ with compact support. Taking $\phi(v)\coloneqq\min\{1,(2-|v-y|/R)_+\}$ and using $|\nabla \gauge|\leq\Lambda_\gauge$ gives, for $y\in\R^d$ and $R>0$,
\begin{equation}\label{eq:hessian-growth}
\mu(B(y,R))\leq2^d\omega_d\Lambda_\gauge R^{d-1}\, .
\end{equation}

\begin{lemma}\label{lem:cell}
There exists $C(d)<\infty$ such that, for every $x\in\Z^d$,
\begin{equation*}
\int_{C_x}|\nabla \gauge(v)-\xi(x)|\dd v\leq C\mu\bigl(B(x,6\sqrt d)\bigr)\, .
\end{equation*}
\end{lemma}

\begin{proof}
The function $h(v)\coloneqq \gauge(v)-\gauge(x)-\xi(x)\cdot(v-x)$ is convex and nonnegative, also for $x=0$ since $\xi(0)=0$. It satisfies $h(x)=0$ and $\nabla h=\nabla \gauge-\xi(x)$ almost everywhere. By convexity along coordinate lines, $|\partial_ih(v)|\leq h(v+e_i)+h(v-e_i)$, and, since $h(x)=0$, convexity gives $h(v)\leq\nabla h(v)\cdot(v-x)$. Let $\phi(v)\coloneqq(4d-|v-x|^2)_+/2$, so that $\nabla\phi(v)=x-v$ on~$B(x,2\sqrt d)$ and $\int\nabla h\cdot\nabla\phi=\int\nabla\gauge\cdot\nabla\phi$. Hence
\begin{equation*}
\int_{C_x}|\nabla h(v)|\dd v\leq\int_{B(x,2\sqrt d)}h(v)\dd v\leq\int_{B(x,2\sqrt d)}\nabla h(v)\cdot(v-x)\dd v=\int\phi\dd\mu\leq2d\,\mu\bigl(B(x,2\sqrt d)\bigr)\, ,
\end{equation*}
which proves the lemma.
\end{proof}

\begin{proof}[Proof of Lemma~\ref{lem:local}]
Let $y\in\Z^d$ with $|y|\leq3n$. We first show the deterministic estimate
\begin{equation}\label{eq:pointwise}
\ell_n(y)=U_{D_n}(y)-\mathcal M_n^y+O(\log n)\, .
\end{equation}
In~\eqref{eq:dynkin}, the summand at $x=y$ vanishes. For $x\neq y$, replacing $\bar\nabla g(x-y)$ by $(\nf{2}{\omega_d})(x-y)|x-y|^{-d}$, and then by its average over~$C_x$, changes the sum by $O\bigl(\sum_{|x|\leq n}(1+|x-y|)^{-d}\bigr)=O(\log n)$ by~\eqref{eq:gradient}. Replacing $\xi(x)$ by $\nabla\gauge$ on~$C_x$ changes it by $O(1)$ for the~$x$ with $|x-y|\leq8\sqrt d$ and, by Lemma~\ref{lem:cell} and the bounded overlap of the balls $B(x,6\sqrt d)$, by at most $C\int_{\{1\leq|v-y|\leq8n\}}|v-y|^{1-d}\dd\mu(v)$ for the others; by~\eqref{eq:hessian-growth} each of the $O(\log n)$ dyadic annuli about~$y$ contributes $O(1)$. This proves~\eqref{eq:pointwise}. Splitting the integrals defining $U_{D_n}(v)$ and $U_{D_n}(v')$ at distance~$2\sqrt d$ from~$v$ gives similarly, for $v,v'\in\R^d$ with $|v-v'|\leq\sqrt d$ and $|v|,|v'|\leq3n$,
\begin{equation}\label{eq:cellmodulus}
|U_{D_n}(v)-U_{D_n}(v')|\leq C\log n\, .
\end{equation}
Every site visited before time~$n$ lies within distance~$4n$ of~$y$, and the weights $(1+|x-y|)^{2-2d}$ over $|x-y|\leq4n$ sum to $O(\log n)$ for $d=2$ and $O(1)$ for $d\geq3$. So~\eqref{eq:bracket} bounds $\langle\mathcal M^y\rangle_n$ by $C\max_x\ell_n(x)$ times this sum, and on the event of~\eqref{eq:localmart}, $|\mathcal M^y_n|$ is at most the right side of~\eqref{eq:approx}. For $y\in\R^d$ with $|y|\leq2n$, apply this and~\eqref{eq:pointwise} at the site~$z$ with $y\in C_z$, and use~\eqref{eq:cellmodulus}.
\end{proof}

\subsection{Bounds on the potential}\label{sec:potential-bounds}

The potential of a set is bounded by a power of its volume, and by a multiple of the width of every region between two balls of~$\gauge$ about the origin that contains the set. Let $D\subset\R^d$ be a bounded measurable set. For every $y\in\R^d$, the integral of $|v-y|^{1-d}$ over~$D$ is at most its integral over the ball of volume~$|D|$ centered at~$y$, which is $C|D|^{\nf{1}{d}}$; since $|\nabla \gauge|\leq\Lambda_\gauge$, there exists $C(d,\gauge)<\infty$ such that
\begin{equation}\label{eq:potential-bound}
\sup_{y\in\R^d}|U_D(y)|\leq C\drift|D|^{\nf{1}{d}}\, .
\end{equation}
The function~$U_D$ is continuous, because it is the convolution of the bounded, compactly supported vector field~$\ind_D\nabla\gauge$ with a kernel that is the sum of an integrable function and a bounded one.

Convexity of~$\gauge$ along lines gives the second bound.

\begin{lemma}\label{lem:layer}
Let $b\geq0$ and $a>0$, and let $D$ be a measurable subset of $\{v\in\R^d:b\leq \gauge(v)\leq b+a\}$. Then
\begin{equation*}
\sup_{y\in\R^d}|U_D(y)|\leq2d\drift a\, .
\end{equation*}
\end{lemma}

\begin{proof}
In polar coordinates $v=y+tu$ about~$y$, as for~\eqref{eq:gradient-identity}, $U_D(y)$ is $2\drift/\omega_d$ times the integral over unit vectors~$u$ of $\int_0^\infty\ind_D(y+tu)h_u'(t)\dd t$, where $h_u(t)\coloneqq\gauge(y+tu)$. The convex function~$h_u$ is monotone on each of two intervals; on each, the substitution $s=h_u(t)$ bounds the integral of $\ind_{[b,b+a]}(h_u(t))|h_u'(t)|$ by~$a$. So the inner integral lies in~$[-a,a]$, and the unit sphere has area~$d\omega_d$.
\end{proof}

\section{Coarse bounds on the range and the local times}\label{sec:coarse}

We prove that, with high probability, $|A_n|$ has order~$r_n^d$, and $\Rout(n)=\max_{j\leq n}|X_j|$ and the largest local time have order~$r_n$ (Proposition~\ref{prop:coarse}), as outlined in step~2 of Section~\ref{sec:strategy}.

\begin{proposition}\label{prop:coarse}
Let $p>0$. There exist $c(d,\drift,\gauge,p)>0$ and $C(d,\drift,\gauge,p)<\infty$ such that, for every integer~$n\geq2$, with probability at least $1-Cn^{-p}$,
\begin{equation*}
cr_n^d\leq|A_n|\leq Cr_n^d,\qquad
cr_n\leq\max_{x\in\Z^d}\ell_n(x)\leq Cr_n,\qquad
cr_n\leq\Rout(n)\leq Cr_n\, .
\end{equation*}
\end{proposition}

\subsection{Crossing a region with small local times}\label{sec:crossing}

Outside a ball $\{\gauge\leq a\}$, the drift at every first departure points toward the ball, by convexity. If the local times outside the ball are at most~$L$, we use the exponential of the distance to the ball, with rate of order $1/(L+1)$, and group its conditional increments by site: far from the ball, the drift of the first departure from a site outweighs the second-order terms of its at most~$L$ departures.

The crossing estimate states that, for every $p>0$, there exists $C(d,\drift,\gauge,p)<\infty$ such that, for every integer $N\geq2$, with probability at least $1-CN^{-p}$, simultaneously for all integers $0\leq t\leq N$ and all real $a\geq0$,
\begin{equation}\label{eq:distance-crossing}
\max_{0\leq j\leq t}\operatorname{dist}(X_j,\{\gauge\leq a\})
\leq C\Bigl(1+\max_{x\in\Z^d,\gauge(x)\geq a}\ell_t(x)\Bigr)\log N\, .
\end{equation}
\begin{proof}[Proof of~\eqref{eq:distance-crossing}]
Fix $a\geq0$, and let $\delta(x)$ be the Euclidean distance from~$x$ to $\{\gauge\leq a\}$. If $\delta(x)\geq1$, let $\bar x$ be the point of $\{\gauge\leq a\}$ nearest to~$x$, and let $\nu(x)\coloneqq(x-\bar x)/\delta(x)$. Bounding $\delta(x+e)$ by $|x+e-\bar x|$ gives $\delta(x+e)\leq\delta(x)+\nu(x)\cdot e+1/(2\delta(x))$ for every unit vector $e\in\Z^d$. The definition of a subgradient gives $\xi(x)\cdot(x-\bar x)\geq\gauge(x)-a$, and the point $ax/\gauge(x)$ of the ball gives $\delta(x)\leq(\gauge(x)-a)/\lambda_\gauge$; so $\xi(x)\cdot\nu(x)\geq \lambda_\gauge$. Let $0<\theta\leq1$ and $f\coloneqq e^{\theta\delta}$. Since $\delta$ is $1$-Lipschitz and $e^u\leq1+u+u^2$ for $|u|\leq1$, at every time~$j$ with $\delta(X_j)\geq1$,
\begin{equation*}
\E\bigl(f(X_{j+1})-f(X_j)\bigm|\mathcal F_j\bigr)
\leq\theta f(X_j)\Bigl(\theta+\frac1{2\delta(X_j)}-\drift \lambda_\gauge I_j\Bigr)\, .
\end{equation*}

Fix an integer $L\geq0$ and a real $s>0$, and let $\theta\coloneqq\drift \lambda_\gauge/(2L+2)$. Let $T$ be the first time $t\leq N$ at which $\delta(X_t)\geq s$ or some site~$x$ with $\gauge(x)\geq a$ has $\ell_t(x)=L+1$, and $T\coloneqq N$ if there is none. A site~$x$ with $\delta(x)\geq(L+1)/(\drift \lambda_\gauge)$, a distance larger than~$1$ by~\eqref{eq:ellipticity}, has at most $L+1$ departures before~$T$, and only the first of them is a first departure. So the bounds above at these departures sum to at most
\begin{equation*}
\theta f(x)\Bigl((L+1)\Bigl(\theta+\frac1{2\delta(x)}\Bigr)-\drift \lambda_\gauge\Bigr)\leq0\, .
\end{equation*}
At the other sites and their neighbors, $\theta\delta\leq1$, so each departure from these sites contributes at most~$e$. Since $f(X_0)=1$, it follows that $\E f(X_T)\leq1+eN$. On the event that $\max_{j\leq t}\delta(X_j)\geq s$ and $\max_{\gauge(x)\geq a}\ell_t(x)\leq L$ for some $t\leq N$, the time~$T$ is the first time at which the walk is at distance at least~$s$ from the ball, so $f(X_T)\geq e^{\theta s}$. By Markov's inequality, this event has probability at most $(1+eN)e^{-\theta s}\leq CN^{-p-2}$ when $s=C(L+1)\log N$ with $C$ large. A union bound over the integers $0\leq L\leq N$ and $0\leq a\leq\Lambda_\gauge N$ proves~\eqref{eq:distance-crossing} for these~$a$, because the local times at time $t\leq N$ are at most~$N$; for integers $a>\Lambda_\gauge N$ the path lies in the ball. For real~$a$, apply the bound with $\lceil a\rceil$ in place of~$a$: this does not increase the maximum of the local times, and, by radial scaling, every point of $\{\gauge\leq\lceil a\rceil\}$ lies within distance $1/\lambda_\gauge$ of $\{\gauge\leq a\}$.
\end{proof}

Since $|v|\leq a/\lambda_\gauge$ on $\{\gauge\leq a\}$, the same event gives, for all $t\leq N$ and $a\geq0$,
\begin{equation}\label{eq:euclid-crossing}
\Rout(t)\leq\frac a{\lambda_\gauge}+C\Bigl(1+\max_{x\in\Z^d,\gauge(x)\geq a}\ell_t(x)\Bigr)\log N\, .
\end{equation}
For the Euclidean norm, $\lambda_\gauge=1$ and the maximum is over the sites with $|x|\geq a$.

\subsection{The largest ball of the norm in~$D_n$}\label{sec:largest-ball}

The following objects are used in Sections~\ref{sec:coarse}--\ref{sec:outer}. Let
\begin{equation*}
b\coloneqq\inf\{\gauge(y):y\in\R^d\setminus D_n\}\qquad\text{and}\qquad
E\coloneqq D_n\setminus\{\gauge<b\}\, ,
\end{equation*}
so that $\{\gauge<b\}$ is the largest ball of~$\gauge$ about the origin contained in~$D_n$. Choose a point $y_0\in\overline{\R^d\setminus D_n}$ with $\gauge(y_0)=b$, the \defn{contact point}, and a cell~$C_z$ disjoint from~$D_n$ whose closure contains~$y_0$. Then $\ell_n(z)=0$ and $\gauge(z)\geq b$. Let
\begin{equation*}
H\coloneqq U_{D_n}(y_0)\qquad\text{and}\qquad \mathcal I\coloneqq\int_E(\gauge(v)-b)\dd v\, .
\end{equation*}
Lemma~\ref{lem:ballpotential} gives $U_{\{\gauge<b\}}(y_0)=0$, so $H=U_E(y_0)$. For almost every $v\in E$, inequality~\eqref{eq:convexity} for the subgradient~$\nabla\gauge(v)$ gives $\nabla \gauge(v)\cdot(v-y_0)\geq \gauge(v)-b\geq0$. Hence $H\geq0$, and, if $E\subset B(y_0,\rho)$ for some $\rho>0$, then
\begin{equation}\label{eq:norm-excess}
H\geq\frac{2\drift}{\omega_d}\int_E\frac{\gauge(v)-b}{|v-y_0|^d}\dd v\geq\frac{c\mathcal I}{\rho^d}\, ,\qquad\text{so}\quad \mathcal I\leq C\rho^dH\, .
\end{equation}
Since $\widetilde\ell_n$ vanishes on~$C_z$ and $U_{D_n}$ is continuous,
\begin{equation}\label{eq:contact-error}
0\leq H\leq\sup_{y\in C_z}\bigl|\widetilde\ell_n(y)-U_{D_n}(y)\bigr|\, .
\end{equation}

The integral~$\mathcal I$ also bounds the potential of~$E$: splitting~$E$ at the level $\gauge=b+\mathcal I^{\nf{1}{(d+1)}}$, Lemma~\ref{lem:layer} and~\eqref{eq:potential-bound} bound the potentials of the two parts by $2d\drift\mathcal I^{\nf{1}{(d+1)}}$ and $C\mathcal I^{\nf{1}{(d+1)}}$, so
\begin{equation}\label{eq:layer-excess}
\sup_{y\in\R^d}|U_E(y)|\leq C\mathcal I^{\nf{1}{(d+1)}}\, ;
\end{equation}
this also holds if $\mathcal I=0$, because then $|E|=0$.

\subsection{Proof of the coarse bounds}

\begin{proof}[Proof of Proposition~\ref{prop:coarse}]
For $R\geq2$, let $\tau_R\coloneqq\inf\{t:\gauge(X_t)\geq R\}$. We prove that, for every $p>0$, with probability at least $1-CR^{-p}$,
\begin{equation*}
cR^{d+1}\leq\tau_R\leq CR^{d+1}
\qquad\text{and}\qquad
\max_x\ell_{\tau_R}(x)\leq CR\, .
\end{equation*}
This implies the proposition: for $R$ a multiple of~$r_n$ with $cR^{d+1}>n$ we get $n<\tau_R$, hence $\Rout(n)\leq Cr_n$ and $\max_x\ell_n(x)\leq Cr_n$, and the remaining bounds follow from $n\leq|A_n|\max_x\ell_n(x)$ and $|A_n|\leq C(1+\Rout(n))^d$.

Let $\rho\coloneqq(2/\lambda_\gauge+2\sqrt d)R$. Since $R\geq2$ and $\{\gauge<R\}\subset B(0,R/\lambda_\gauge)$, the ball $B(0,\rho/2)$ contains every cell within distance~$\sqrt d$ of~$\{\gauge<R\}$. Before~$\tau_R$ the walk departs only from sites in $\{\gauge<R\}$, so~\eqref{eq:dynkin} and~\eqref{eq:gradient} give $\ell_t(y)\leq|\mathcal M^y_t|+CR$ for $t\leq\tau_R$. Stop the martingales~$\mathcal M^y$ for $y\in\Z^d\cap B(0,\rho)$ at~$\tau_R$ and when one of them first reaches~$R$ in absolute value. Until then every local time is at most~$CR$, so the elapsed time is at most~$CR^{d+1}$ and, by~\eqref{eq:bracket}, every stopped bracket is at most $CR\log R$. Lemma~\ref{lem:freedman} at time~$\lceil CR^{d+1}\rceil$ and a union bound over the $O(R^d)$ sites~$y$ show that, except with probability~$CR^{-p}$, the stopped martingales stay below $C\sqrt R\log R<R$. On this event, no stopped martingale reaches~$R$ in absolute value before~$\tau_R$, which gives the upper bounds.

For the lower bound we work on this event and take $R$ large. Let $N\coloneqq\lceil CR^{d+1}\rceil$, with the constant of the upper bound on~$\tau_R$. Except with probability $CR^{-p}$, the event of~\eqref{eq:distance-crossing} also holds, so~\eqref{eq:euclid-crossing} holds at time $t=\tau_R\leq N$, with $\log N\leq C\log R$. The comparison~\eqref{eq:pointwise} at time~$\tau_R$, whose error terms depend only on the departed sites, and the bound $|\mathcal M^y_{\tau_R}|\leq C\sqrt R\log R$ give $|\widetilde\ell_{\tau_R}-U_{D_{\tau_R}}|\leq C\sqrt R\log R$ on $B(0,\rho/2)$. For the objects of Section~\ref{sec:largest-ball} at time~$\tau_R$, the bounds~\eqref{eq:contact-error}, \eqref{eq:norm-excess} with this~$\rho$ and~\eqref{eq:layer-excess} give $H\leq C\sqrt R\log R$, $\mathcal I\leq CR^dH$ and $\sup|U_E|\leq C(R^{d+\nf12}\log R)^{\nf{1}{(d+1)}}=o(R/\log R)$. Since $U_{D_{\tau_R}}=2d\drift(b-\gauge)_++U_E$ by Lemma~\ref{lem:ballpotential}, every site of~$A_{\tau_R}$ outside $\{\gauge<b\}$ has local time $o(R/\log R)$, and~\eqref{eq:euclid-crossing} with $a=b$ gives $R/\Lambda_\gauge\leq\Rout(\tau_R)\leq b/\lambda_\gauge+o(R)$, so $b\geq cR$. Then $\widetilde\ell_{\tau_R}\geq d\drift b-o(R)$ on $\{\gauge\leq b/2\}$, and integrating gives $\tau_R\geq cR^{d+1}$.
\end{proof}

\section{The limit shape}\label{sec:norms}

We prove Theorem~\ref{thm:norm-shape} with rates (Proposition~\ref{prop:norm-shape}), as outlined in step~3 of Section~\ref{sec:strategy}.

\begin{proposition}\label{prop:norm-shape}
Let $\gauge$ be a norm, let $\xi(x)\in\partial \gauge(x)$ be subgradients chosen at the sites $x\in\Z^d\setminus\{0\}$, and let $\drift>0$ satisfy~\eqref{eq:ellipticity}. For every $p>0$ there exists $C(d,\drift,\gauge,p)<\infty$, not depending on the choice of the subgradients, such that, for every integer~$n\geq2$, with probability at least $1-Cn^{-p}$ the following hold simultaneously.
\begin{enumerate}[label=\textup{(\roman*)}]
\item \underline{\emph{Radii}}:
\begin{equation}\label{eq:norm-shape}
\Bigl|\inf_{y\notin D_n}\gauge(y)-r_n\Bigr|\leq C\begin{cases}r_n^{\nf34}(\log n)^{\nf14},&d=2,\\(r_n\log n)^{\nf12},&d\geq3\end{cases}
\end{equation}
and
\begin{equation}\label{eq:norm-outer}
\max_{0\leq j\leq n}\gauge(X_j)-r_n\leq C\begin{cases}r_n^{\nf56}(\log n)^{\nf76},&d=2,\\r_n^{\nf{d}{(d+1)}}(\log n)^{\nf{(d+2)}{(d+1)}},&d\geq3\end{cases}\, .
\end{equation}
\item \underline{\emph{Local times}}:
\begin{equation}\label{eq:norm-local}
\max_{x\in\Z^d}|\ell_n(x)-2d\drift(r_n-\gauge(x))_+|\leq C\begin{cases}r_n^{\nf56}(\log n)^{\nf16},&d=2,\\r_n^{\nf{d}{(d+1)}}(\log n)^{\nf{1}{(d+1)}},&d\geq3\end{cases}\, .
\end{equation}
\end{enumerate}
\end{proposition}

Throughout the rest of the paper let
\begin{equation}\label{eq:qn}
q_n\coloneqq\begin{dcases}\sqrt{\frac{\log n}{r_n}},&d=2,\\\frac{\log n}{r_n},&d\geq3\end{dcases}\, .
\end{equation}
The bounds in~\eqref{eq:norm-shape}, \eqref{eq:norm-outer} and~\eqref{eq:norm-local} are $Cr_nq_n^{\nf12}$, $Cr_nq_n^{\nf{1}{(d+1)}}\log n$ and $Cr_nq_n^{\nf{1}{(d+1)}}$.

\subsection{Martingale estimates}\label{sec:norm-event}

We collect on one event the martingale bounds used in Sections~\ref{sec:norms}--\ref{sec:outer}. Fix $p>0$, and let $K$ be a constant with $\Rout(n)\leq(K-1)r_n$ on the event of Proposition~\ref{prop:coarse}, so that $D_n\subset B(0,Kr_n)$ there for large~$n$. The process
\begin{equation*}
\mathcal Q_t\coloneqq\sum_{j<t}2X_j\cdot\bigl(X_{j+1}-X_j+\drift I_j\xi(X_j)\bigr)
\end{equation*}
is a martingale. Since every step has squared length one, Euler's relation gives
\begin{equation}\label{eq:quadratic}
|X_n|^2=n-2\drift\sum_{x\in A_n}\gauge(x)+\mathcal Q_n\, .
\end{equation}
Let $\widehat{\mathcal Q}$ be~$\mathcal Q$ stopped when the walk leaves~$B(0,Kr_n)$. Its increments are at most~$Cr_n$, so by Azuma's inequality $|\widehat{\mathcal Q}_n|\leq Cr_n\sqrt{n\log n}$ with probability at least $1-Cn^{-p}$, and on the event of Proposition~\ref{prop:coarse} we have $\widehat{\mathcal Q}_n=\mathcal Q_n$, so that
\begin{equation}\label{eq:quadraticerror}
|\mathcal Q_n|\leq Cr_n\sqrt{n\log n}\, .
\end{equation}

To bound the brackets~$\langle\mathcal M^y\rangle_n$ in terms of~$\ell_n(y)$, for $y\in\Z^d$ let $f_y\coloneqq(g_y-\gamma)^2$, where $\gamma\coloneqq0$ if $d\geq3$ and $\gamma\coloneqq(\nf2\pi)\log r_n$ if $d=2$. If $|y|\leq3n$, the increments of~$\mathcal M^{f_y}$ through time~$n$ are at most~$C\bar g$, where $\bar g\coloneqq1$ if $d\geq3$ and $\bar g\coloneqq\log n$ if $d=2$, and Lemma~\ref{lem:freedman} applies to $\mathcal M^{f_y}/(C\bar g)$; we take a union bound over these~$y$.

The event of Section~\ref{sec:norm-event} is the intersection of the events of Proposition~\ref{prop:coarse} and Lemma~\ref{lem:local} with those on which these bounds on~$\widehat{\mathcal Q}_n$ and~$\mathcal M^{f_y}_n$, \eqref{eq:localmart}, and~\eqref{eq:distance-crossing} with $N=n$ hold. The complement of this event has probability at most $Cn^{-p}$, and on the event the arguments of Sections~\ref{sec:norms}--\ref{sec:outer} are deterministic. On the event, $\max_x\ell_n(x)\leq Cr_n$, so~\eqref{eq:approx} gives, for $y\in\R^d$ with $|y|\leq2n$,
\begin{equation}\label{eq:norm-approx}
|\widetilde\ell_n(y)-U_{D_n}(y)|\leq C\begin{cases}\sqrt{r_n}\log n,&d=2,\\\sqrt{r_n\log n},&d\geq3\end{cases}\, .
\end{equation}

\subsection{The radii and the local times}\label{sec:norm-inner}

Let $b$, $E$, $y_0$, $z$, $H$ and~$\mathcal I$ be as in Section~\ref{sec:largest-ball}. Then $|v-y_0|\leq Cr_n$ for $v\in E$, because $D_n\subset B(0,Kr_n)$ and $\gauge(y_0)=b\leq Cr_n$, so~\eqref{eq:norm-excess} gives $\mathcal I\leq Cr_n^dH$. By~\eqref{eq:contact-error} and~\eqref{eq:norm-approx}, $H\leq Cr_n^{\nf12}\log n$.

Lemma~\ref{lem:contact} improves this bound, because the walk has not departed from~$z$: up to $C\log n$, Dynkin's formula at~$z$ bounds~$H$ by~$|\mathcal M^z_n|$. Dynkin's formula for~$f_z$ bounds the bracket of~$\mathcal M^z$ in terms of~$\ell_n(z)$, which vanishes. The same argument at every site~$y$ bounds $|\ell_n(y)-U_{D_n}(y)|$ in terms of~$\ell_n(y)$, which gives~\eqref{eq:square-one-sided}, used again in Sections~\ref{sec:contact} and~\ref{sec:outer}.

\begin{lemma}\label{lem:contact}
There exists $C(d,\drift,\gauge,p)<\infty$ such that, on the event of Section~\ref{sec:norm-event}, for large~$n$,
\begin{equation*}
H\leq Cr_nq_n=C\begin{cases}\sqrt{r_n\log n},&d=2,\\\log n,&d\geq3\end{cases}\, .
\end{equation*}
\end{lemma}

\begin{proof}
We first bound $|\ell_n(y)-U_{D_n}(y)|$ in terms of~$\ell_n(y)$ for every $y\in\Z^d$ with $|y|\leq3n$. For such~$y$, let
\begin{equation*}
\mathcal E_y\coloneqq\frac1{2d}\sum_{x\in\Z^d}\ell_n(x)\sum_{|e|=1}\bigl(g_y(x+e)-g_y(x)\bigr)^2\geq\frac{\langle\mathcal M^y\rangle_n}{2d}\, .
\end{equation*}
Since $\Delta g_y=\ind_{\{y\}}$, the Laplacian of~$f_y$ at~$x$ is $(2d)^{-1}\sum_{|e|=1}(g_y(x+e)-g_y(x))^2+2(g(0)-\gamma)\ind_{\{y\}}(x)$, so Dynkin's formula for~$f_y$, grouped by site as in~\eqref{eq:dynkin}, gives
\begin{equation*}
\mathcal E_y=2\bigl(\gamma-g(0)\bigr)\ell_n(y)+f_y(X_n)-f_y(0)-\mathcal M^{f_y}_n+\drift\sum_{x\in A_n}\xi(x)\cdot\bar\nabla f_y(x)\, .
\end{equation*}
Since $|g_y-\gamma|\leq C\bar g$ within distance~$4n$ of~$y$, it follows that $\langle\mathcal M^{f_y}\rangle_n\leq C\bar g^2\mathcal E_y$ and that the first three terms are $O(\bar g\ell_n(y)+\bar g^2)$. By~\eqref{eq:kernel-asymptotics} and~\eqref{eq:gradient}, the term of~$x$ in the last sum is $O((1+|x-y|)^{3-2d})$ if $d\geq3$ and $O\bigl((1+|x-y|)^{-1}(1+|\log((1+|x-y|)/r_n)|)\bigr)$ if $d=2$. These bounds decrease in~$|x-y|$, so their sum over the at most $Cr_n^d$ sites of~$A_n$ is $O(\log n)$ if $d\geq3$ and $O\bigl(\int_1^{Cr_n}(1+|\log(s/r_n)|)\dd s\bigr)=O(r_n)$ if $d=2$. Without the shift~$\gamma$, the planar sum would be of order $r_n\log r_n$. The bound of Lemma~\ref{lem:freedman} for~$\mathcal M^{f_y}$ and Young's inequality then give, since $(\log n)^3=o(r_n)$,
\begin{equation*}
\mathcal E_y\leq C\begin{cases}\ell_n(y)\log n+r_n,&d=2,\\\ell_n(y)+\log n,&d\geq3\end{cases}\, ,
\end{equation*}
and with~\eqref{eq:localmart} and~\eqref{eq:pointwise}, simultaneously for all $y\in\Z^d$ with $|y|\leq3n$,
\begin{equation*}
|\ell_n(y)-U_{D_n}(y)|\leq C\begin{cases}\sqrt{r_n\log n}+(\log n)\sqrt{\ell_n(y)},&d=2,\\\log n+\sqrt{\ell_n(y)\log n},&d\geq3\end{cases}\, .
\end{equation*}
Young's inequality bounds the term containing~$\ell_n(y)$ on the right side by $\ell_n(y)/2+Cr_nq_n$, because $(\log n)^2=o(\sqrt{r_n\log n})$ when $d=2$; so $U_{D_n}(y)\geq\ell_n(y)/2-Cr_nq_n$, and $|U_{D_n}(y)|\leq Cr_nq_n$ if $y\notin A_n$. By~\eqref{eq:cellmodulus}, and since $\widetilde\ell_n$ vanishes and $|U_{D_n}|\leq C$ outside~$B(0,2n)$, it follows that
\begin{equation}\label{eq:square-one-sided}
U_{D_n}\geq\frac12\widetilde\ell_n-Cr_nq_n\quad\text{on }\R^d
\qquad\text{and}\qquad
|U_{D_n}|\leq Cr_nq_n\quad\text{on }\R^d\setminus D_n\, .
\end{equation}
Then \eqref{eq:contact-error} and~\eqref{eq:square-one-sided} give $H\leq Cr_nq_n$.
\end{proof}

\begin{proof}[Proof of Proposition~\ref{prop:norm-shape}]
We work on the event of Section~\ref{sec:norm-event} and take $n$ large. By Lemma~\ref{lem:contact}, $\mathcal I\leq Cr_n^dH\leq Cr_n^{d+1}q_n$. Splitting~$E$ at the level $\gauge=b+r_nq_n^{\nf12}$, into $E\cap\{\gauge<b+r_nq_n^{\nf12}\}$, of volume $O(r_n^dq_n^{\nf12})$, and a set of volume at most~$\mathcal I/(r_nq_n^{\nf12})$, gives $|E|\leq Cr_n^dq_n^{\nf12}$. By~\eqref{eq:quadratic}, \eqref{eq:quadraticerror} and $|X_n|\leq Cr_n$, we have $n-2\drift\sum_{x\in A_n}\gauge(x)=O(r_n^{\nf{(d+3)}{2}}(\log n)^{\nf12})$; the sum differs from $\int_{D_n}\gauge$ by $O(r_n^d)$, and $\int_{D_n}\gauge=d|B_\gauge|b^{d+1}/(d+1)+O(r_n|E|)$. Dividing by $n=2d\drift|B_\gauge|r_n^{d+1}/(d+1)$ gives
\begin{equation}\label{eq:radius-volume}
\left|1-\left(\frac b{r_n}\right)^{d+1}\right|
\leq C\bigl(r_n^{\nf{(1-d)}{2}}(\log n)^{\nf12}+r_n^{-1}+|E|r_n^{-d}\bigr)\, ,
\end{equation}
which is at most $Cq_n^{\nf12}$; so $|b-r_n|\leq Cr_nq_n^{\nf12}$, which is~\eqref{eq:norm-shape}. By Lemma~\ref{lem:ballpotential}, $U_{D_n}=2d\drift(b-\gauge)_++U_E$, and~\eqref{eq:layer-excess} gives $\sup|U_E|\leq C(r_n^{d+1}q_n)^{\nf{1}{(d+1)}}=Cr_nq_n^{\nf{1}{(d+1)}}$. With~\eqref{eq:norm-approx} and the bound on $|b-r_n|$, this gives $|\widetilde\ell_n-2d\drift(r_n-\gauge)_+|\leq Cr_nq_n^{\nf{1}{(d+1)}}$ on~$\R^d$, which is~\eqref{eq:norm-local}. Finally, this bound gives $\ell_n(x)\leq Cr_nq_n^{\nf{1}{(d+1)}}$ at the sites with $\gauge(x)\geq r_n$, so~\eqref{eq:distance-crossing} with $a=r_n$ and $t=N=n$, together with $\gauge(x)-r_n\leq\Lambda_\gauge\operatorname{dist}(x,\{\gauge\leq r_n\})$, proves~\eqref{eq:norm-outer}.
\end{proof}

\begin{proof}[Proof of Theorem~\ref{thm:norm-shape}]
Apply Proposition~\ref{prop:norm-shape} with $p>1$ and the Borel--Cantelli lemma. Its bounds are $o(r_n)$, so~\eqref{eq:norm-shape} and~\eqref{eq:norm-outer} give~(i), and~\eqref{eq:norm-local} gives~(ii). At every fixed site~$x$, part~(ii) gives $\ell_n(x)/r_n\to2d\drift>0$, proving~(iii). Taking the Euclidean norm gives Theorem~\ref{thm:shape}.
\end{proof}

The proofs in Sections~\ref{sec:norm-setup}--\ref{sec:norms} use the symmetry $\gauge(-x)=\gauge(x)$ only to bound the coordinates of the subgradients. Let $B\subset\R^d$ be a compact convex set with the origin in its interior, and let $\gauge_B(x)\coloneqq\inf\{t>0:x\in tB\}$ be its Minkowski functional, which is convex and positively homogeneous and satisfies $c|x|\leq\gauge_B(x)\leq C|x|$. Every $\xi\in\partial\gauge_B(x)$ satisfies $-\gauge_B(-e_i)\leq\xi_i\leq\gauge_B(e_i)$; so, if $\drift\max\{\gauge_B(e_i),\gauge_B(-e_i)\}<\nf1d$ for $1\leq i\leq d$, then the probabilities~\eqref{eq:kernel} are positive, and the proofs of Theorem~\ref{thm:norm-shape} and Proposition~\ref{prop:norm-shape} apply without change to~$\gauge_B$, with $|B|$ in place of~$|B_\gauge|$. Thus every such set~$B$ is the limit shape of the centrally excited random walk with $\gauge_B$ in place of~$\gauge$.

\section{The inner radius for the Euclidean norm}\label{sec:contact}

We prove the bounds on $|\Rin(n)-r_n|$ and on the volume in Theorem~\ref{thm:fluctuations}\,(i) and~(ii), and a first uniform bound on the local times (Proposition~\ref{prop:inner}), as outlined in step~4 of Section~\ref{sec:strategy}.

\begin{proposition}[Inner radius]\label{prop:inner}
Let $\gauge$ be the Euclidean norm, and let $0<\drift<\nf1d$. For every $p>0$ there exists $C(d,\drift,p)<\infty$ such that, for every integer~$n\geq2$, with probability at least $1-Cn^{-p}$, the following hold simultaneously, where $q_n$ is defined in~\eqref{eq:qn}.
\begin{enumerate}[label=\textup{(\roman*)}]
\item \underline{\emph{Inner radius}}:
\begin{equation}\label{eq:inradius}
|\Rin(n)-r_n|\leq Cr_nq_n\, .
\end{equation}
\item \underline{\emph{Volume}}:
\begin{equation*}
\bigl|r_n^{-1}D_n\mathbin{\triangle}B(0,1)\bigr|\leq Cq_n\, .
\end{equation*}
\item \underline{\emph{Local times}}:
\begin{equation*}
\sup_{y\in\R^d}\bigl|\widetilde\ell_n(y)-2d\drift(r_n-|y|)_+\bigr|\leq Cr_nq_n^{\nf1d}\, .
\end{equation*}
\end{enumerate}
\end{proposition}

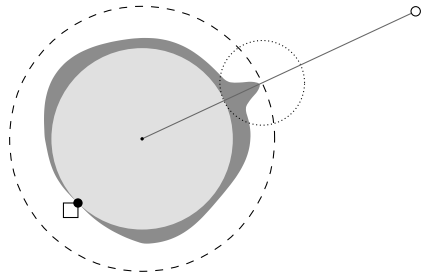
\begin{figure}[b]
\centering
\begin{tikzpicture}[scale=0.8]
\fill[black!45] plot[smooth cycle,tension=0.6] coordinates {(1.790,0.000) (1.792,0.078) (1.787,0.156) (1.775,0.234) (1.759,0.310) (1.737,0.385) (1.718,0.460) (1.725,0.544) (1.795,0.653) (1.909,0.791) (1.941,0.905) (1.811,0.943) (1.612,0.931) (1.466,0.934) (1.383,0.968) (1.327,1.018) (1.276,1.070) (1.225,1.122) (1.173,1.173) (1.120,1.223) (1.066,1.271) (1.011,1.317) (0.954,1.362) (0.895,1.405) (0.835,1.446) (0.772,1.484) (0.708,1.519) (0.642,1.550) (0.574,1.578) (0.505,1.602) (0.435,1.623) (0.363,1.639) (0.291,1.652) (0.219,1.661) (0.146,1.666) (0.073,1.668) (0.000,1.668) (-0.073,1.664) (-0.145,1.659) (-0.217,1.651) (-0.289,1.641) (-0.361,1.630) (-0.433,1.617) (-0.505,1.603) (-0.577,1.586) (-0.650,1.568) (-0.722,1.548) (-0.794,1.525) (-0.865,1.499) (-0.936,1.469) (-1.006,1.437) (-1.074,1.400) (-1.140,1.359) (-1.203,1.313) (-1.264,1.264) (-1.320,1.209) (-1.372,1.151) (-1.419,1.089) (-1.462,1.023) (-1.499,0.955) (-1.530,0.884) (-1.557,0.810) (-1.578,0.736) (-1.595,0.660) (-1.607,0.585) (-1.615,0.509) (-1.620,0.434) (-1.622,0.360) (-1.622,0.286) (-1.620,0.213) (-1.617,0.141) (-1.614,0.070) (-1.607,0.000) (-1.594,-0.070) (-1.581,-0.138) (-1.567,-0.206) (-1.552,-0.274) (-1.535,-0.340) (-1.516,-0.406) (-1.493,-0.471) (-1.468,-0.534) (-1.439,-0.596) (-1.406,-0.656) (-1.370,-0.713) (-1.332,-0.769) (-1.291,-0.822) (-1.247,-0.873) (-1.202,-0.922) (-1.156,-0.970) (-1.108,-1.015) (-1.061,-1.061) (-1.015,-1.108) (-0.970,-1.156) (-0.923,-1.203) (-0.874,-1.249) (-0.824,-1.293) (-0.772,-1.336) (-0.717,-1.378) (-0.661,-1.418) (-0.603,-1.456) (-0.543,-1.493) (-0.482,-1.528) (-0.418,-1.561) (-0.353,-1.593) (-0.286,-1.624) (-0.218,-1.652) (-0.147,-1.679) (-0.074,-1.704) (-0.000,-1.726) (0.076,-1.733) (0.151,-1.729) (0.226,-1.720) (0.301,-1.707) (0.375,-1.689) (0.447,-1.667) (0.517,-1.640) (0.585,-1.609) (0.652,-1.573) (0.715,-1.534) (0.776,-1.491) (0.835,-1.446) (0.891,-1.398) (0.944,-1.349) (0.996,-1.298) (1.045,-1.246) (1.094,-1.193) (1.141,-1.141) (1.187,-1.088) (1.233,-1.034) (1.278,-0.981) (1.324,-0.927) (1.369,-0.872) (1.414,-0.816) (1.459,-0.760) (1.503,-0.701) (1.547,-0.641) (1.588,-0.578) (1.628,-0.513) (1.664,-0.446) (1.698,-0.376) (1.727,-0.304) (1.751,-0.231) (1.770,-0.155) (1.783,-0.078)};
\fill[black!12] (0,0) circle (1.500);
\draw[dashed] (0,0) circle (2.191);
\draw[densely dotted] (1.986,0.926) circle (0.700);
\draw[thin,black!60] (0,0) -- (4.524,2.109);
\filldraw[fill=white,draw=black] (-1.301,-1.301) rectangle (-1.061,-1.061);
\fill (-1.061,-1.061) circle (2.2pt);
\fill (0,0) circle (0.8pt);
\filldraw[fill=white,draw=black] (4.524,2.109) circle (2.2pt);
\end{tikzpicture}
\caption{The ball $B(0,b)$ (light), $E=D_n\setminus B(0,b)$ (dark) and the circle of radius $\Rout(n)+2\sqrt d$ (dashed). Points of~$E$ contribute nonnegatively to~$U_E$ at the contact point~$y_0$ (black dot), beside a cell disjoint from~$D_n$ (square), and those in the dotted ball contribute negatively at the open dot.}
\label{fig:contact}
\end{figure}

Fix $p>0$. In this section and the next we take $n$ large and work on the event of Section~\ref{sec:norm-event}, whose complement has probability at most $Cn^{-p}$ and on which $D_n\subset B(0,Kr_n)$, with $K$ as in Section~\ref{sec:norm-event}. Let $b\coloneqq\Rin(n)$ and $E\coloneqq D_n\setminus B(0,b)$, and let $y_0$ and~$H$ be as in Section~\ref{sec:largest-ball} for the Euclidean norm, so that $|y_0|=b$ and $H=U_{D_n}(y_0)=U_E(y_0)$ (Figure~\ref{fig:contact}).

If $s=|v|\geq b=|y_0|$, then
\begin{equation}\label{eq:contact-sign}
\frac{u_v\cdot(v-y_0)}{|v-y_0|^d}
=\frac{s^2-v\cdot y_0}{s|v-y_0|^d}
\geq\frac1{2s}|v-y_0|^{2-d}
\geq2^{1-d}s^{1-d}\, .
\end{equation}
The first inequality uses $2(s^2-v\cdot y_0)=|v-y_0|^2+s^2-b^2$, and the second $|v-y_0|\leq2s$ and $d\geq2$. Since $|v|\leq Cr_n$ on~$E$, integrating~\eqref{eq:contact-sign} over~$E$ and Lemma~\ref{lem:contact} give
\begin{equation}\label{eq:mass}
|E|\leq Cr_n^{d-1}H\leq Cr_n^dq_n\, .
\end{equation}

\begin{proof}[Proof of Proposition~\ref{prop:inner}]
Substituting~\eqref{eq:mass} into~\eqref{eq:radius-volume} gives $|b-r_n|\leq Cr_nq_n$, which is~(i), and (i) and~\eqref{eq:mass} give~(ii). On~$B(0,Kr_n)$, the decomposition $U_{D_n}(y)=2d\drift(b-|y|)_++U_E(y)$, the bound~\eqref{eq:potential-bound}, \eqref{eq:norm-approx} and~(i) give
\begin{equation*}
\bigl|\widetilde\ell_n(y)-2d\drift(r_n-|y|)_+\bigr|
\leq C\bigl(\sqrt{r_n}\log n+r_nq_n+|E|^{\nf1d}\bigr)\leq Cr_nq_n^{\nf1d}\, ,
\end{equation*}
and outside this ball both terms on the left vanish; this proves~(iii).
\end{proof}

\section{The outer radius and the local times}\label{sec:outer}

We prove the outer-radius and local-time bounds in Theorem~\ref{thm:fluctuations}, as outlined in step~5 of Section~\ref{sec:strategy}. The crossing estimate bounds $\Rout(n)-\Rin(n)$ by $C\log n$ times one plus the largest local time outside~$B(0,\Rin(n))$, and~\eqref{eq:square-one-sided} bounds these local times by the positive contributions to~$U_E$. Section~\ref{sec:outer-potential} bounds these contributions at points beyond the sphere of radius $\Rout(n)+2\sqrt d$, using its curvature (Figure~\ref{fig:contact}).

\begin{proposition}[Outer radius]\label{prop:stronger-outer}
Let $\gauge$ be the Euclidean norm, and let $0<\drift<\nf{1}{d}$. For every $p>0$ there exists $C(d,\drift,p)<\infty$ such that, for every integer~$n\geq2$, with probability at least $1-Cn^{-p}$,
\begin{equation*}
\Rout(n)\leq r_n+C\begin{cases}\sqrt{r_n}(\log n)^{\nf52},&d=2,\\(\log n)^{d+1},&d\geq3\end{cases}\, .
\end{equation*}
\end{proposition}

Let
\begin{equation*}
w\coloneqq\Rout(n)-b+2\sqrt d
\qquad\text{and}\qquad
L\coloneqq1+\max\{\ell_n(x):x\in\Z^d,\ |x|\geq b\}\, .
\end{equation*}
Since $|v|\leq\Rout(n)+\sqrt d/2$ on~$D_n$, we have $b\leq\Rout(n)+\sqrt d/2$, so $w\geq3\sqrt d/2$ and $E\subset\{b\leq|v|\leq b+w\}$. By~\eqref{eq:norm-outer} and~\eqref{eq:inradius}, $w=o(r_n)$, so $1\leq w\leq b$ and $r_n/2\leq b\leq2r_n$ for large~$n$. The crossing estimate~\eqref{eq:euclid-crossing} with $t=N=n$ and $a=b$ gives
\begin{equation}\label{eq:outer-width-occupation}
w\leq CL\log n\, .
\end{equation}

\subsection{The potential beyond the inner radius}\label{sec:outer-potential}

For a bounded measurable set~$D\subset\R^d$ and $y\in\R^d$, write
\begin{equation*}
U_D^+(y)\coloneqq\frac{2\drift}{\omega_d}\int_D\left[u_v\cdot\frac{v-y}{|v-y|^d}\right]_+\dd v
\end{equation*}
for the potential with its negative contributions discarded. If $|y|\geq b+w$, then every $v\in E$ satisfies $|v|\leq|y|$, so the identity $2(|v|^2-v\cdot y)=|v-y|^2+|v|^2-|y|^2$ gives $u_v\cdot(v-y)\leq|v-y|^2/(2|v|)$. Hence, with
\begin{equation*}
F(y)\coloneqq\int_E\frac{|v-y|^{2-d}}{|v|}\dd v\, ,
\end{equation*}
every $y\in\R^d$ with $|y|\geq b+w$ satisfies
\begin{equation}\label{eq:curvature-size}
U_E^+(y)\leq\frac\drift{\omega_d}F(y)\, .
\end{equation}
We prove that
\begin{equation}\label{eq:curvature-improved}
\sup_{\R^d}F\leq Cr_nq_n\, .
\end{equation}

\begin{proof}[Proof of~\eqref{eq:curvature-improved}]
If $d=2$, then $F\leq|E|/b$, and~\eqref{eq:mass} gives the bound. If $d\geq3$, then differentiating $t^{d-2}|v-tu|^{2-d}$ in~$t$ gives, for every unit vector~$u$,
\begin{equation*}
\frac{\dd}{\dd t}\bigl(t^{d-2}F(tu)\bigr)=\frac{(d-2)\omega_d}{2\drift}t^{d-3}U_E(tu)\, .
\end{equation*}
If $t\geq b$, then \eqref{eq:square-one-sided} gives $U_E(tu)=U_{D_n}(tu)\geq-Cr_nq_n$. By~\eqref{eq:mass}, $F(2(b+w)u)\leq|E|(b+w)^{2-d}/b\leq Cr_nq_n$. Integrating over $[t,2(b+w)]$ gives $t^{d-2}F(tu)\leq Cr_n^{d-1}q_n$. Since $b\geq r_n/2$, this proves~\eqref{eq:curvature-improved} on $\{b\leq|y|\leq2(b+w)\}$, and the maximum principle for~$F$, which is harmonic off the closure of~$E$ and tends to zero at infinity, extends the bound to~$\R^d$.
\end{proof}

\subsection{The outer radius}

By~\eqref{eq:square-one-sided}, a bound on~$U_E^+$ on the annulus $\{b\leq|y|\leq b+w\}$ bounds~$L$; the following deterministic lemma gives such a bound.

\begin{lemma}\label{lem:near-far}
There exists $C(d,\drift)<\infty$ with the following property. Let $1\leq w\leq b$ and $\lambda\geq1$, and let $D$ be a measurable subset of $\{v\in\R^d:b\leq|v|\leq b+w\}$ such that $|D\cap B((b+w)u,t)|\leq\lambda t^{d-1}$ for every unit vector~$u$ and every $t\geq w$, and $\int_D|v-y|^{2-d}\dd v\leq\lambda b$ for every $y\in\R^d$. Then every $y\in\R^d$ with $b\leq|y|\leq b+w$ satisfies
\begin{equation*}
U_D^+(y)\leq C\bigl(\lambda w^{d-1}\bigr)^{\nf1d}+C\lambda\, .
\end{equation*}
\end{lemma}

\begin{proof}
Since $B(y,t)\subset B((b+w)u_y,2t)$ for $t\geq w$, the first hypothesis gives $|D\cap B(y,t)|\leq2^{d-1}\lambda t^{d-1}$. So the points of~$D$ within distance~$w$ of~$y$ contribute at most $C(\lambda w^{d-1})^{\nf1d}$, as in~\eqref{eq:potential-bound}. For the other points, the identity used for~\eqref{eq:curvature-size} and $|v|^2-|y|^2\leq3bw$ bound the kernel by $2w|v-y|^{-d}+|v-y|^{2-d}/(2b)$. By the second hypothesis, the second term contributes at most $C\lambda$, and, integrating by parts in~$|v-y|$, the first contributes at most $Cw\int_w^\infty\lambda t^{d-1}t^{-d-1}\dd t\leq C\lambda$.
\end{proof}

\begin{proof}[Proof of Proposition~\ref{prop:stronger-outer}]
We first show that, for every unit vector~$u$ and every $t\geq w$,
\begin{equation}\label{eq:outer-cap-mass}
|E\cap B((b+w)u,t)|\leq Cr_nq_n t^{d-1}\, .
\end{equation}
For $t>\nf{r_n}4$ this follows from~\eqref{eq:mass}. For $w\leq t\leq\nf{r_n}4$, let $y_1\coloneqq(b+w+4t)u$ (Figure~\ref{fig:contact}). Then $y_1\notin D_n$, because $\Rout(n)=b+w-2\sqrt d$, and $U_{B(0,b)}(y_1)=0$ because $|y_1|\geq b$, so~\eqref{eq:square-one-sided} gives $U_E(y_1)\geq-Cr_nq_n$. If $v\in E\cap B((b+w)u,t)$, then $u_v\cdot u\geq1-t/(b+w)$ and $t\leq(b+w)/2$, so $u_v\cdot(v-y_1)\leq(b+w)-(b+w+4t)(1-t/(b+w))\leq-t$, while $|v-y_1|\leq5t$; so the kernel of~$U_E(y_1)$ is at most $-ct^{1-d}$ at these points. By~\eqref{eq:curvature-size} and~\eqref{eq:curvature-improved}, the other points contribute at most $Cr_nq_n$, which proves~\eqref{eq:outer-cap-mass}.

Since $|v|\leq2b$ on~$E$, inequality~\eqref{eq:curvature-improved} gives $\int_E|v-y|^{2-d}\dd v\leq2bF(y)\leq Cbr_nq_n$. With~\eqref{eq:outer-cap-mass}, Lemma~\ref{lem:near-far} applies to $D=E$ and $\lambda=Cr_nq_n$, and gives, for $b\leq|y|\leq b+w$,
\begin{equation}\label{eq:positive-bound}
U_E^+(y)\leq C\bigl(r_nq_n w^{d-1}\bigr)^{\nf1d}+Cr_nq_n\, .
\end{equation}
If $|y|\geq b$, then \eqref{eq:square-one-sided} gives $\widetilde\ell_n(y)\leq2U_E^+(y)+Cr_nq_n$, so $L\leq C((r_nq_nw^{d-1})^{\nf1d}+r_nq_n)$. With~\eqref{eq:outer-width-occupation} and Young's inequality, $w\leq\nf w2+Cr_nq_n(\log n)^d$, that is,
\begin{equation}\label{eq:outer-reduction}
w\leq Cr_nq_n(\log n)^d=C\begin{cases}\sqrt{r_n}(\log n)^{\nf52},&d=2,\\(\log n)^{d+1},&d\geq3\end{cases}\, .
\end{equation}
Since $\Rout(n)\leq b+w$ and $|b-r_n|\leq Cr_nq_n$, this proves the proposition.
\end{proof}

\subsection{The uniform bound on the local times}\label{sec:global-profile}

We prove Theorem~\ref{thm:fluctuations}\,(iii) by bounding~$U_E$ on the annulus $\{b\leq|y|\leq b+w\}$, which contains~$E$. There $U_{B(0,b)}$ vanishes, so~\eqref{eq:square-one-sided} gives $U_E=U_{D_n}\geq-Cr_nq_n$, and $U_E\leq U_E^+$ and~\eqref{eq:positive-bound} give the upper bound. The function~$U_E$ is continuous, harmonic off the closure of~$E$, and tends to zero at infinity, so the maximum principle extends both bounds to $B(0,b)$ and to the complement of $B(0,b+w)$. With~\eqref{eq:outer-reduction}, this gives
\begin{equation*}
\sup_{\R^d}|U_E|\leq C\bigl((r_nq_n w^{d-1})^{\nf1d}+r_nq_n\bigr)
\leq Cr_nq_n(\log n)^{d-1}
=C\begin{cases}\sqrt{r_n}(\log n)^{\nf32},&d=2,\\(\log n)^d,&d\geq3\end{cases}\, .
\end{equation*}
Inserting this bound in place of~\eqref{eq:potential-bound} in the proof of Proposition~\ref{prop:inner}\,(iii) gives Theorem~\ref{thm:fluctuations}\,(iii).

\begin{proof}[Proof of Theorem~\ref{thm:fluctuations}]
On the event of Section~\ref{sec:norm-event}, for large~$n$, Proposition~\ref{prop:inner} gives the bounds on~$\Rin(n)$ and on the volume, Proposition~\ref{prop:stronger-outer} the bound on~$\Rout(n)$, and Section~\ref{sec:global-profile} the bound on the local times. The Borel--Cantelli lemma with $p>1$ gives the almost-sure statement.
\end{proof}

\section{Limit laws}\label{sec:sharp}

We use the martingale central limit theorem and the law of the iterated logarithm in the following form, here and in Section~\ref{sec:lower}. Let $Z$ be a real square-integrable martingale with $Z_0=0$, and let $s_n\to\infty$ be deterministic with $\log\log s_n\sim\log\log n$. Suppose that $\langle Z\rangle_n/s_n^2\to\sigma^2$ almost surely for a constant $\sigma\geq0$ and that $|Z_{j+1}-Z_j|\leq\varepsilon_j$ for $\mathcal F_j$-measurable~$\varepsilon_j$ with $\max_{j<n}\varepsilon_j=o(s_n(\log\log n)^{-\nf12})$ almost surely. Then $Z_n/s_n$ converges in distribution to $\mathcal N(0,\sigma^2)$ \citep[Corollary~3.1]{HallHeyde1980}, and, if $\sigma>0$, then $\limsup_{n\to\infty}\pm Z_n/(s_n\sqrt{2\log\log n})=\sigma$ almost surely for each sign \citep[Theorems~1 and~2]{Stout1970}. For a vector martingale satisfying these hypotheses in every direction with the same~$s_n$, the Cram\'er--Wold device gives a joint Gaussian limit.

\subsection{The sum of the distances from the origin}

For the Euclidean norm, \eqref{eq:quadratic} reads
\begin{equation}\label{eq:moment-martingale}
\mathcal Q_n=2\drift\sum_{x\in A_n}|x|-n+|X_n|^2\, .
\end{equation}
Let $\Rmom(n)$ be the radius of the ball about the origin on which the integral of~$|v|$ equals $\sum_{x\in A_n}|x|$, that is,
\begin{equation}\label{eq:radial-moment-def}
\Rmom(n)\coloneqq\left(\frac{d+1}{d\omega_d}\sum_{x\in A_n}|x|\right)^{\nf{1}{(d+1)}}\, .
\end{equation}

\begin{theorem}[Sum of the distances from the origin]\label{thm:moment-fluctuations}
Let $\gauge$ be the Euclidean norm, let $0<\drift<\nf1d$, and let $\Rmom(n)$ be as in~\eqref{eq:radial-moment-def}. Then the following hold.
\begin{enumerate}[label=\textup{(\roman*)}]
\item \underline{\emph{Central limit theorem}}: as $n\to\infty$,
\begin{equation*}
\frac{\sum_{x\in A_n}|x|-n/(2\drift)}{r_n^{\nf{(d+3)}{2}}}
\xrightarrow{\mathrm d}\mathcal N\Bigl(0,\frac{2d\omega_d}{\drift(d+2)(d+3)}\Bigr)
\end{equation*}
and
\begin{equation*}
\frac{\Rmom(n)-r_n}{r_n^{\nf{(3-d)}{2}}}
\xrightarrow{\mathrm d}\mathcal N\Bigl(0,\frac{2}{\drift d\omega_d(d+2)(d+3)}\Bigr)\, .
\end{equation*}
\item \underline{\emph{Law of the iterated logarithm}}: for each choice of sign, almost surely,
\begin{align}
\limsup_{n\to\infty}\frac{\pm\bigl(\sum_{x\in A_n}|x|-n/(2\drift)\bigr)}{r_n^{\nf{(d+3)}{2}}\sqrt{2\log\log n}}
&=\Bigl(\frac{2d\omega_d}{\drift(d+2)(d+3)}\Bigr)^{\nf12}\, ,\notag\\
\limsup_{n\to\infty}\frac{\pm(\Rmom(n)-r_n)}{r_n^{\nf{(3-d)}{2}}\sqrt{2\log\log n}}
&=\Bigl(\frac{2}{\drift d\omega_d(d+2)(d+3)}\Bigr)^{\nf12}\, .\label{eq:moment-radius-lil}
\end{align}
\end{enumerate}
\end{theorem}

\begin{proof}[Proof of Theorem~\ref{thm:moment-fluctuations}]
We first compute the limit of $\langle\mathcal Q\rangle_n/r_n^{d+3}$. Under both~\eqref{eq:kernel} and the simple random walk transition probabilities, the conditional second-moment matrix of $X_{j+1}-X_j$ is $d^{-1}$ times the identity, so the conditional variance of $\zeta\cdot(X_{j+1}-X_j)$ is $|\zeta|^2/d-\drift^2I_j(\zeta\cdot u_{X_j})^2$ for $\mathcal F_j$-measurable~$\zeta$. With $\zeta=2X_j$ this gives
\begin{equation}\label{eq:moment-exact-bracket}
\langle\mathcal Q\rangle_n=\frac4d\sum_{x\in\Z^d}\ell_n(x)|x|^2-4\drift^2\sum_{x\in A_n}|x|^2\, .
\end{equation}
By Theorem~\ref{thm:shape}, the second sum is $O(r_n^{d+2})$ and the first is $r_n^{d+3}\bigl(2d\drift\int_{|v|<1}(1-|v|)|v|^2\dd v+o(1)\bigr)$, almost surely, so
\begin{equation*}
\frac{\langle\mathcal Q\rangle_n}{r_n^{d+3}}\longrightarrow\frac{8d\drift\omega_d}{(d+2)(d+3)}
\qquad\text{almost surely}\, .
\end{equation*}
Since $|\mathcal Q_{j+1}-\mathcal Q_j|\leq4|X_j|=O(r_n)$, the limit theorems stated at the beginning of this section apply to~$\mathcal Q$ with $s_n=r_n^{\nf{(d+3)}2}$, and the endpoint term $|X_n|^2=O(r_n^2)$ in~\eqref{eq:moment-martingale} is negligible; this gives the assertions for the sum. The derivative of $((d+1)t/(d\omega_d))^{\nf{1}{(d+1)}}$ at $t=n/(2\drift)$ is $1/(d\omega_dr_n^d)$, which gives the assertions for~$\Rmom(n)$.
\end{proof}

\subsection{The potential at a fixed site}

Since $B(0,\Rin(n))\subset D_n\subset B(0,\Rout(n)+\sqrt d)$, Theorem~\ref{thm:fluctuations}\,(i) and~(ii) give, almost surely for all large~$n$, the bound $|D_n\mathbin{\triangle}B(0,r_n)|\leq Cr_n^dq_n$, with $q_n$ as in~\eqref{eq:qn}, and, for every $v\in D_n\mathbin{\triangle}B(0,r_n)$,
\begin{equation}\label{eq:centering-shape-input}
\bigl||v|-r_n\bigr|\leq C\begin{cases}\sqrt{r_n}(\log n)^{\nf52},&d=2,\\(\log n)^{d+1},&d\geq3\end{cases}\, .
\end{equation}
For a fixed site~$y$, on the annulus defined by~\eqref{eq:centering-shape-input} the kernel of~$U_{D_n}(y)$ is close to $r_n^{-d}|v|$, which relates $U_{D_n}(y)$ to $\sum_{x\in A_n}|x|$ and hence, by~\eqref{eq:moment-martingale}, to~$\mathcal Q_n$.

\begin{lemma}\label{lem:fixed-site-centering}
Let $\gauge$ be the Euclidean norm, let $0<\drift<\nf1d$, and let $y\in\Z^d$. There exists $C(d,\drift,y)<\infty$ such that, almost surely, for all sufficiently large~$n$,
\begin{equation}\label{eq:fixed-site-centering}
\left|U_{D_n}(y)-2d\drift(r_n-|y|)-\frac{\mathcal Q_n}{\omega_dr_n^d}\right|\leq C\begin{cases}(\log n)^3,&d=2,\\1,&d\geq3\end{cases}\, ,
\end{equation}
where $\mathcal Q$ is the martingale of~\eqref{eq:quadratic}.
\end{lemma}

\begin{proof}
Let $r\coloneqq r_n$ and let $a$ be the right side of~\eqref{eq:centering-shape-input}. For fixed~$y$ and $v\in D_n\mathbin{\triangle}B(0,r)$, the kernel $u_v\cdot(v-y)|v-y|^{-d}$ equals $|v|^{1-d}+O_y(r^{-d})=r^{-d}|v|+O_y(r^{-d}(1+a))$. Integrating against $\ind_{D_n}-\ind_{B(0,r)}$, whose $L^1$ norm is $O(r^dq_n)$, and using~\eqref{eq:ballpotential-euclid} and $\int_{B(0,r)}|v|\dd v=n/(2\drift)$, gives
\begin{equation*}
U_{D_n}(y)-2d\drift(r-|y|)
=\frac{2\drift}{\omega_dr^d}
\left(\int_{D_n}|v|\dd v-\frac n{2\drift}\right)
+O_y\bigl(q_n(1+a)\bigr)\, .
\end{equation*}
Replacing the integral by $\sum_{x\in A_n}|x|$ changes it by $O(r^d)$, so by~\eqref{eq:moment-martingale} the main term is $\mathcal Q_n/(\omega_dr^d)+O(1)$, and $q_n(1+a)$ is $O((\log n)^3)$ in the plane and $O(1)$ in higher dimensions.
\end{proof}

\subsection{Local times at fixed sites}

Up to a smaller correction at first departures, the bracket of~$\mathcal M^y$ is a sum of local times with weights that decay like~$|x-y|^{2-2d}$; in the plane each dyadic annulus about~$y$ inside~$B(0,r_n)$ contributes order~$r_n$, which gives an extra factor~$\log r_n$. For functions $f_1,f_2\colon\Z^d\to\R$ with bounded nearest-neighbor increments, let
\begin{equation}\label{eq:site-gamma}
\Gamma(f_1,f_2)(x)\coloneqq\frac1{2d}\sum_{|e|=1}\bigl(f_1(x+e)-f_1(x)\bigr)\bigl(f_2(x+e)-f_2(x)\bigr)-\Delta f_1(x)\Delta f_2(x)\, ,
\end{equation}
which is the conditional covariance of the increments of~$\mathcal M^{f_1}$ and~$\mathcal M^{f_2}$ at a departure from~$x$ other than the first. As in Section~\ref{sec:dynkin}, $g_y=g(\cdot-y)$ and $G$ is the Green function.

\begin{theorem}[Local times at fixed sites]\label{thm:site-fluctuations}
Let $\gauge$ be the Euclidean norm, and let $0<\drift<\nf1d$. Then the following hold.
\begin{enumerate}[label=\textup{(\roman*)}]
\item \underline{\emph{Central limit theorem}}: for every $y\in\Z^d$, as $n\to\infty$,
\begin{equation}\label{eq:site-clt}
\begin{aligned}
\frac{\ell_n(y)-2d\drift(r_n-|y|)}{\sqrt{r_n\log r_n}}&\xrightarrow{\mathrm d}\mathcal N\Bigl(0,\frac{16\drift}{\pi}\Bigr)&&\text{if } d=2\, ,\\
\frac{\ell_n(y)-2d\drift(r_n-|y|)}{\sqrt{r_n}}&\xrightarrow{\mathrm d}\mathcal N\bigl(0,2d\drift(2G(0)-1)\bigr)&&\text{if } d\geq3\, .
\end{aligned}
\end{equation}
\item \underline{\emph{Law of the iterated logarithm}}: for every $y\in\Z^d$ and each choice of sign, almost surely,
\begin{equation*}
\begin{aligned}
\limsup_{n\to\infty}\frac{\pm(\ell_n(y)-2d\drift(r_n-|y|))}{\sqrt{2r_n\log r_n\log\log n}}&=\Bigl(\frac{16\drift}{\pi}\Bigr)^{\nf12}&&\text{if } d=2\, ,\\
\limsup_{n\to\infty}\frac{\pm(\ell_n(y)-2d\drift(r_n-|y|))}{\sqrt{2r_n\log\log n}}&=\bigl(2d\drift(2G(0)-1)\bigr)^{\nf12}&&\text{if } d\geq3\, .
\end{aligned}
\end{equation*}
\item \underline{\emph{Joint limits}}: for all $y_1,\ldots,y_k\in\Z^d$, the convergence in~\eqref{eq:site-clt} holds jointly for $y=y_1,\ldots,y_k$, and the limit is a centered Gaussian vector whose covariance matrix has $(i,j)$ entry
\begin{equation}\label{eq:site-covariance}
\begin{cases}
16\drift/\pi,&d=2,\\
2d\drift\bigl(2G(y_i-y_j)-\ind_{\{y_i=y_j\}}\bigr),&d\geq3
\end{cases}\, .
\end{equation}
\end{enumerate}
\end{theorem}

\begin{proof}[Proof of Theorem~\ref{thm:site-fluctuations}]
We first compute the limits of the brackets $\langle\mathcal M^y,\mathcal M^z\rangle_n$. By~\eqref{eq:gradient}, a first departure from~$x$ changes the conditional covariance from $\Gamma(g_y,g_z)(x)$ by $O_{y,z}((1+|x|)^{2-2d})$, and $A_n\subset B(0,Cr_n)$ eventually almost surely; thus
\begin{equation}\label{eq:site-bracket}
\langle\mathcal M^y,\mathcal M^z\rangle_n
=\sum_{x\in\Z^d}\ell_n(x)\Gamma(g_y,g_z)(x)
+\begin{cases}O_{y,z}(\log r_n),&d=2,\\O_{y,z}(1),&d\geq3\end{cases}\, .
\end{equation}
The identity $\Gamma(f_1,f_2)=\Delta(f_1f_2)-f_1\Delta f_2-f_2\Delta f_1-\Delta f_1\Delta f_2$ and summation by parts give
\begin{equation}\label{eq:gamma-parts}
\sum_x\Gamma(f_1,f_2)(x)=-\sum_x\bigl(f_1\Delta f_2+f_2\Delta f_1+\Delta f_1\Delta f_2\bigr)(x)
\end{equation}
when the boundary terms vanish. If $d\geq3$, then they vanish for $f_1=g_y$ and $f_2=g_z$ by~\eqref{eq:kernel-asymptotics} and~\eqref{eq:gradient}, and since $\Delta g_y=\ind_{\{y\}}$ the sum is $2G(y-z)-\ind_{\{y=z\}}$. The weights are summable, so Theorem~\ref{thm:shape}\,(ii) and dominated convergence give $\langle\mathcal M^y,\mathcal M^z\rangle_n/r_n\to2d\drift(2G(y-z)-\ind_{\{y=z\}})$ almost surely. If $d=2$, then $\Gamma(g_y,g_z)(x)=2/(\pi^2|x|^2)+O_{y,z}(|x|^{-3})$, so $\sum_{|x|\leq r}\Gamma(g_y,g_z)(x)=(4/\pi)\log r+O_{y,z}(1)$. In~\eqref{eq:site-bracket} the term $4\drift r_n$ of the cone $4\drift(r_n-|x|)_+$ contributes $(16\drift/\pi)r_n\log r_n+O(r_n)$, the term $-4\drift|x|$ contributes $O(r_n)$, and the error $o(r_n)$ of Theorem~\ref{thm:shape}\,(ii) contributes $o(r_n\log r_n)$; so $\langle\mathcal M^y,\mathcal M^z\rangle_n/(r_n\log r_n)\to16\drift/\pi$ almost surely.

The martingales~$\mathcal M^y$ have bounded increments, so the limit theorems stated at the beginning of this section apply to $(-\mathcal M^{y_i})_{i\leq k}$, with $s_n=\sqrt{r_n\log r_n}$ if $d=2$ and $s_n=\sqrt{r_n}$ if $d\geq3$. By~\eqref{eq:pointwise} and Lemma~\ref{lem:fixed-site-centering}, $\ell_n(y)-2d\drift(r_n-|y|)=-\mathcal M_n^y+\mathcal Q_n/(\omega_dr_n^d)+O_y((\log n)^3)$, and Theorem~\ref{thm:moment-fluctuations} makes the last two terms negligible.
\end{proof}

\section{Lower bounds on the fluctuations}\label{sec:lower}

We prove Theorem~\ref{thm:sharp}, an upper bound on the deviations of the local times away from the boundary (Proposition~\ref{prop:bulk-profile}), and logarithmic lower bounds on the radii (Proposition~\ref{prop:log-lower}). As outlined in Section~\ref{sec:strategy}, we obtain the lower bounds in parts~(i), (iii) and~(iv)(a) from exponential martingales (Lemma~\ref{lem:exp-deviation}): for parts~(i) and~(iv)(a) by the Paley--Zygmund inequality for one exponential martingale, and for part~(iii) by a second-moment bound for the average of the exponential martingales at many well-separated sites. Part~(iii) needs many sites because, by Theorem~\ref{thm:site-fluctuations}, the deviation at one fixed site is typically smaller than the bound in part~(iii) by a factor of order~$(\log n)^{\nf12}$.

Throughout this section $0<\drift<\nf1d$, and $\Omega_n$ is the event of Section~\ref{sec:norm-event} with $p=10$, so $\P(\Omega_n^c)=O(n^{-10})$ and the estimates of Sections~\ref{sec:local}--\ref{sec:outer} used below hold on~$\Omega_n$ for large~$n$. In particular, by Section~\ref{sec:global-profile}, in every dimension,
\begin{equation}\label{eq:lower-local-error}
\max_{x\in\Z^d}\bigl|\ell_n(x)-2d\drift(r_n-|x|)_+\bigr|\leq C\sqrt{r_n}(\log n)^{\nf32}\, .
\end{equation}

\begin{lemma}\label{lem:exp-deviation}
Let $0<c_0\leq C_0$. There exist $c(c_0,C_0)>0$ and $C(c_0,C_0)<\infty$ with the following property. Let $m$ and~$n$ be positive integers, let $b>0$, $\delta\geq0$ and $0\leq\alpha\leq1$, and let $S^1,\ldots,S^m$ be real martingales with respect to a common filtration, with $S^i_0=0$ and $|S^i_{t+1}-S^i_t|\leq b$ for $0\leq t<n$. Let $\Omega$ be an event with $\P(\Omega)\geq1-\alpha$ on which $c_0\leq\langle S^i\rangle_n\leq C_0$ for $1\leq i\leq m$. Then, for every real $\beta\geq C$ with $\beta b\leq c$, the following hold.
\begin{enumerate}[label=\textup{(\roman*)}]
\item For $1\leq i\leq m$,
\begin{equation*}
\P(S^i_n\geq c\beta)\geq\frac14e^{-C\beta^2}-\alpha
\qquad\text{and}\qquad
\P(S^i_n\leq-c\beta)\geq\frac14e^{-C\beta^2}-\alpha\, .
\end{equation*}
\item If $|\langle S^i,S^j\rangle_n|\leq\delta$ on~$\Omega$ for all $i\neq j$, and $\beta^2\delta+\beta^3b\leq1$, then
\begin{equation*}
\P\Bigl(\max_{1\leq i\leq m}S^i_n<c\beta\Bigr)\leq C\bigl(m^{-1}e^{C\beta^2}+\beta^2\delta+\beta^3b+e^{C\beta^2}\sqrt\alpha\bigr)\, .
\end{equation*}
\end{enumerate}
\end{lemma}

\begin{proof}
We take $c\leq\nf18$ small and $C$ large, depending only on $c_0$ and~$C_0$. Stop all the martingales just before one of the brackets would exceed~$2C_0$. Brackets are predictable, so this is a stopping time, and on~$\Omega$ it does not stop before~$n$; so we may assume that $\langle S^i\rangle_n\leq2C_0$, and then $V_{ij}\coloneqq\langle S^i,S^j\rangle_n$ satisfies $|V_{ij}|\leq2C_0$. For a martingale~$S$ with respect to the common filtration~$(\mathcal G_t)$, with $S_0=0$ and increments at most~$b_1$, and for real~$s$, let $\psi(s)$ be the sum over $t<n$ of the logarithms of the conditional moment generating functions $\E(e^{s(S_{t+1}-S_t)}\mid\mathcal G_t)$, and let $W(s)\coloneqq\exp\{sS_n-\psi(s)\}$, so that $\E W(s)=1$. If $\langle S\rangle_n\leq8C_0$ and $|s|b_1\leq\nf14$, then a conditional Taylor expansion gives $\psi(s)=s^2\langle S\rangle_n/2+O(|s|^3b_1)$, and the same at~$2s$; so $W(s)^2=W(2s)\exp\{\psi(2s)-2\psi(s)\}$ gives $\E W(s)^2\leq e^{Cs^2}$. Apply these bounds to~$S^i$ and $S^i+S^j$, and write $W_i$ and~$\psi_i$, and $W_{ij}$ and~$\psi_{ij}$, for the corresponding values at $s=\beta$. Then
\begin{equation*}
\psi_i=\frac{\beta^2}2V_{ii}+O(\beta^3b)
\qquad\text{and}\qquad
\psi_{ij}-\psi_i-\psi_j=\beta^2V_{ij}+O(\beta^3b)\, .
\end{equation*}

(i) On~$\Omega$ the bound $V_{ii}\geq c_0$ gives $\psi_i\geq c_0\beta^2/4$, so $W_i\geq\nf12$ implies $S^i_n\geq c_0\beta/4-(\log 2)/\beta\geq c\beta$. By the Paley--Zygmund inequality, $\P(W_i\geq\nf12)\geq(\E W_i)^2/(4\E W_i^2)\geq e^{-C\beta^2}/4$. The same argument applies to~$-S^i$.

(ii) For $i\neq j$, the exponent in $W_iW_j=W_{ij}\exp\{\psi_{ij}-\psi_i-\psi_j\}$ is at most $\beta^2\delta+C\beta^3b$ on~$\Omega$ and at most $C\beta^2$ everywhere, so the Cauchy--Schwarz inequality on~$\Omega^c$ gives $\E(W_iW_j)\leq1+C(\beta^2\delta+\beta^3b)+e^{C\beta^2}\sqrt\alpha$. Hence the average~$\overline W$ of $W_1,\ldots,W_m$ has mean~$1$ and variance at most $C(m^{-1}e^{C\beta^2}+\beta^2\delta+\beta^3b+e^{C\beta^2}\sqrt\alpha)$. On~$\Omega$ the inequality $\overline W\geq\nf12$ implies $W_i\geq\nf12$ for some~$i$, hence $\max_iS^i_n\geq c\beta$ as in~(i), and Chebyshev's inequality proves~(ii).
\end{proof}

\subsection{Lower bounds on the radii and their difference in the plane}

\begin{proof}[Proof of Theorem~\ref{thm:sharp}\,\textup{(i)} and~\textup{(ii)}]
Comparison with unit cells gives $\sum_{|x|<\rho}|x|=d\omega_d\rho^{d+1}/(d+1)+O(\rho^d)$. Since every lattice site in $B(0,\Rin(n))$ belongs to~$A_n$ and $A_n\subset B(0,\Rout(n)+1)$, taking $(d+1)$st roots gives, for every path,
\begin{equation}\label{eq:moment-radius-comparison}
\Rin(n)\leq\Rmom(n)+C
\qquad\text{and}\qquad
\Rout(n)\geq\Rmom(n)-C\, .
\end{equation}
For $d=2$, the two signs in~\eqref{eq:moment-radius-lil} therefore give~\eqref{eq:planar-lil-lower}, with constant $\sqrt2/\sqrt{20\pi\drift}=1/\sqrt{10\pi\drift}$.

For part~(i), let $r\coloneqq r_n$. On~$\Omega_n$ the stopped martingale~$\widehat{\mathcal Q}$ of Section~\ref{sec:norm-event}, whose increments are at most $4Kr$, agrees with~$\mathcal Q$ through time~$n$, and Proposition~\ref{prop:coarse} forces at least $n/2$ departures outside $B(0,cr)$; since $\drift<\nf12$, the identity~\eqref{eq:moment-exact-bracket} gives $c_0r^5\leq\langle\widehat{\mathcal Q}\rangle_n\leq C_0r^5$ on~$\Omega_n$. Lemma~\ref{lem:exp-deviation}\,(i) for $\widehat{\mathcal Q}/r^{\nf52}$, with $b=4Kr^{-\nf32}$, $\alpha=\P(\Omega_n^c)$ and $\beta=a_p\sqrt{\log n}$ for a small $a_p>0$, shows that the events $\{\pm\mathcal Q_n\geq ca_pr^{\nf52}\sqrt{\log n}\}\cap\Omega_n$ have probability at least~$n^{-p}$. On~$\Omega_n$ we have $|X_n|^2=O(r^2)$ and $\Rmom(n)=O(r)$, so~\eqref{eq:moment-martingale} and $\Rmom(n)^3-r^3=(3/(2\pi))(\sum_{x\in A_n}|x|-n/(2\drift))$ turn these events into $\pm(\Rmom(n)-r)\geq c\sqrt{r\log n}$, and~\eqref{eq:moment-radius-comparison} proves~\eqref{eq:planar-polynomial-lower}.
\end{proof}

\begin{proof}[Proof of Theorem~\ref{thm:sharp}\,\textup{(iv)}]
Let $d=2$ and $Z_n\coloneqq X_n+\drift\sum_{x\in A_n}u_x$, which is a martingale because the step from~$X_j$ has conditional mean $-\drift I_ju_{X_j}$. Its increments are bounded by~$2$, and the conditional variance computed for~\eqref{eq:moment-exact-bracket} gives, for every unit vector~$e\in\R^2$,
\begin{equation}\label{eq:width-bracket}
\langle e\cdot Z\rangle_n=\frac n2-\drift^2\sum_{x\in A_n}(e\cdot u_x)^2\, ,
\end{equation}
so $Y\coloneqq e_1\cdot Z$ satisfies $(\nf12-\drift^2)n\leq\langle Y\rangle_n\leq\nf n2$ on every path. Suppose that $\nf{r_n}2\leq\Rin(n)\leq\Rout(n)\leq2r_n$. The vectors $u_x$ at the lattice sites with $|x|<\Rin(n)$ sum to zero. Comparing the cells of the other sites of~$A_n$ with the annulus $\{\Rin(n)-1\leq|v|\leq\Rout(n)+1\}$, on which $\int(e\cdot u_v)_+\dd v=\Rout(n)^2-\Rin(n)^2+O(r_n)$ for every unit vector~$e$, gives
\begin{equation}\label{eq:width-annulus}
\Bigl|\sum_{x\in A_n}u_x\Bigr|\leq\Rout(n)^2-\Rin(n)^2+Cr_n\, .
\end{equation}

Lemma~\ref{lem:exp-deviation}\,(i) for $Y/\sqrt n$, with $\alpha=0$ and $\beta=a_p\sqrt{\log n}$ for a small $a_p>0$, shows that $\{Y_n\geq ca_p\sqrt{n\log n}\}\cap\Omega_n$ has probability at least~$n^{-p}$. On this event, for large~$n$, both radii lie in $[\nf{r_n}2,2r_n]$. The bound $|X_n|=O(r_n)$, inequality~\eqref{eq:width-annulus} and the identity $n=(4\pi\drift/3)r_n^3$ then give $\Rout(n)-\Rin(n)\geq c\sqrt{n\log n}/r_n-C\geq c\sqrt{r_n\log n}$, which proves~\eqref{eq:planar-width-polynomial}.

By Theorem~\ref{thm:shape}, $|A_n|=o(n)$, $X_n=o(\sqrt n)$ and $\Rin(n)+\Rout(n)\sim2r_n$ almost surely. So $\langle e\cdot Z\rangle_n/n\to\nf12$ for every unit vector~$e$ by~\eqref{eq:width-bracket}, and the limit theorems stated at the beginning of Section~\ref{sec:sharp}, with $s_n=\sqrt n$, give $\limsup Y_n/\sqrt{n\log\log n}=1$ almost surely and $Z_n/\sqrt n\to N\sim\mathcal N(0,\nf12 I_2)$ in distribution. By~\eqref{eq:width-annulus}, almost surely
\begin{equation*}
\frac{|Z_n|}{\sqrt n}\leq\sqrt{\frac{3\drift}\pi}\,\frac{\Rout(n)-\Rin(n)}{\sqrt{r_n}}(1+o(1))+o(1)\, .
\end{equation*}
The first limit therefore gives $\limsup(\Rout(n)-\Rin(n))/\sqrt{r_n\log\log n}\geq\sqrt{\pi/(3\drift)}$, and, since $\P(|N|\leq t)=1-e^{-t^2}$, the second shows that $\limsup\P(\Rout(n)-\Rin(n)\leq a\sqrt{r_n})\leq1-e^{-3\drift a^2/\pi}$.
\end{proof}

\subsection{Upper bounds on the deviations of the local times away from the boundary}

In the plane the following bound improves Theorem~\ref{thm:fluctuations}\,(iii).

\begin{proposition}\label{prop:bulk-profile}
Let $\gauge$ be the Euclidean norm, let $0<\drift<\nf1d$, and let $0<\theta<1$ and $p>0$. There exists $C(d,\drift,\theta,p)<\infty$ such that, for every integer~$n\geq2$, with probability at least $1-Cn^{-p}$,
\begin{equation}\label{eq:bulk-profile}
\max_{\substack{y\in\Z^d\\|y|\leq\theta r_n}}\bigl|\ell_n(y)-2d\drift(r_n-|y|)\bigr|
\leq C\begin{cases}\sqrt{r_n}\log n,&d=2,\\\sqrt{r_n\log n},&d\geq3\end{cases}\, .
\end{equation}
The same bound holds for all sufficiently large~$n$ almost surely.
\end{proposition}

\begin{proof}
Intersect the events of Theorem~\ref{thm:fluctuations} and of Section~\ref{sec:norm-event}, with exponent $\max\{p,2\}$ in both, and let $E=D_n\setminus B(0,\Rin(n))$ as in Section~\ref{sec:contact}. For large~$n$, every~$y$ with $|y|\leq\theta r_n$ lies at distance at least $(1-\theta)r_n/2$ from~$E$, and $|E|\leq Cr_n^dq_n$, so $|U_E(y)|\leq C_\theta|E|r_n^{1-d}\leq C_\theta r_nq_n$. With $|\Rin(n)-r_n|\leq Cr_nq_n$ and~\eqref{eq:ballpotential-euclid}, this gives, uniformly over $|y|\leq\theta r_n$,
\begin{equation}\label{eq:bulk-potential-only}
\bigl|U_{D_n}(y)-2d\drift(r_n-|y|)\bigr|\leq C_\theta r_nq_n\, ,
\end{equation}
and $r_nq_n$ is at most the right side of~\eqref{eq:norm-approx}, which proves~\eqref{eq:bulk-profile}. The Borel--Cantelli lemma gives the almost-sure statement.
\end{proof}

\subsection{Lower bounds on the deviations of the local times near the origin}\label{sec:sharp-bulk}

By~\eqref{eq:pointwise} and~\eqref{eq:bulk-potential-only}, it suffices to find, among about~$r_n^{\nf18}$ sites at mutual distances at least~$r_n^{\nf14}$, one at which $-\mathcal M_n^y$ exceeds a multiple of $(\log r_n)^{\nf12}$ standard deviations. In the plane every entry of~\eqref{eq:site-covariance} equals $16\drift/\pi$, so the limits at different sites coincide, and we use differences of martingales at two sites at distance about~$r_n^{\nf18}$. The normalized martingales then have brackets bounded above and below and cross brackets of order at most~$r_n^{-\nf14}$ (Lemma~\ref{lem:separated-brackets}), and Lemma~\ref{lem:exp-deviation}\,(ii) applies.

Recall the martingales~$\mathcal M^f$ of Section~\ref{sec:dynkin} and the function~$\Gamma$ of~\eqref{eq:site-gamma}. Let $m\coloneqq\lfloor r_n^{\nf18}\rfloor$, let $s_n\coloneqq(r_n\log r_n)^{\nf12}$ if $d=2$ and $s_n\coloneqq r_n^{\nf12}$ if $d\geq3$, as in the proof of Theorem~\ref{thm:site-fluctuations}, and for $1\leq i\leq m$ let
\begin{equation*}
y_i\coloneqq i\lceil r_n^{\nf14}\rceil e_1\, ,\qquad
\varphi_i\coloneqq\begin{cases}g_{y_i+me_1}-g_{y_i},&d=2,\\g_{y_i},&d\geq3\end{cases}
\qquad\text{and}\qquad
S^i\coloneqq-\frac{\mathcal M^{\varphi_i}}{s_n}\, .
\end{equation*}
Then $|y_i|\leq m\lceil r_n^{\nf14}\rceil\leq r_n^{\nf12}/2$ for large~$n$.

\begin{lemma}\label{lem:separated-brackets}
Let $\gauge$ be the Euclidean norm, and let $0<\drift<\nf1d$. There exist $0<c_0(d,\drift)\leq C_0(d,\drift)$ and $C(d,\drift)<\infty$ such that, for all sufficiently large~$n$, on~$\Omega_n$, for $1\leq i\leq m$ and $1\leq j\leq m$ with $j\neq i$,
\begin{equation*}
c_0\leq\langle S^i\rangle_n\leq C_0
\qquad\text{and}\qquad
|\langle S^i,S^j\rangle_n|\leq Cr_n^{-\nf14}\, .
\end{equation*}
\end{lemma}

\begin{proof}
Let $r\coloneqq r_n$, so that $s_n^2/r$ is $\log r$ if $d=2$ and $1$ if $d\geq3$. By~\eqref{eq:kernel-asymptotics} and~\eqref{eq:gradient}, the increments of~$\varphi_i$ at~$x$ are $O((1+|x-y_i|)^{1-d})$ if $d\geq3$, and in the plane $O((1+|x-y_i|)^{-1}+(1+|x-y_i-me_1|)^{-1})$ for $|x-y_i|\leq3m$ and $O(m|x-y_i|^{-2})$ otherwise. Since $2|ab|\leq a^2+b^2$ and $\Delta \varphi_i$ is supported at the poles, which lie in $B(0,\sqrt r)$,
\begin{equation*}
\sum_x|\Gamma(\varphi_i,\varphi_j)(x)|\leq\frac{Cs_n^2}{r}
\qquad\text{and}\qquad
\sum_{|x|>2\sqrt r}|\Gamma(\varphi_i,\varphi_j)(x)|\leq\frac{Cs_n^2}{r^{\nf32}}\, ,
\end{equation*}
and the first-departure corrections total $O(s_n^2/r)$. Let $\Sigma_{ij}\coloneqq(2d\drift r/s_n^2)\sum_x\Gamma(\varphi_i,\varphi_j)(x)$. On~$\Omega_n$ the bound~\eqref{eq:lower-local-error} gives $\ell_n(x)=2d\drift r+O(\sqrt r(\log n)^{\nf32})$ for $|x|\leq2\sqrt r$, and $\ell_n\leq Cr$ everywhere, so, uniformly in~$i,j$,
\begin{equation*}
\bigl|\langle S^i,S^j\rangle_n-\Sigma_{ij}\bigr|
\leq\frac{C}{r}\bigl(1+\sqrt r(\log n)^{\nf32}\bigr)
\leq Cr^{-\nf12}(\log n)^{\nf32}=o(r^{-\nf14})\, .
\end{equation*}
For $d\geq3$, as in the proof of Theorem~\ref{thm:site-fluctuations}, $\Sigma_{ij}=2d\drift(2G(y_i-y_j)-\ind_{\{i=j\}})$, whose diagonal is $2d\drift(2G(0)-1)>0$ and whose off-diagonal entries are $O(r^{\nf{(2-d)}4})$. In the plane, \eqref{eq:gamma-parts} applies to $\varphi_i$ and~$\varphi_j$, whose values and increments at distance~$R$ are $O(m/R)$ and $O(m/R^2)$; since $\Delta \varphi_i$ and~$\Delta \varphi_j$ have disjoint supports for $i\neq j$, it gives, with $z=y_i-y_j$,
\begin{equation*}
\sum_x\Gamma(\varphi_i,\varphi_j)(x)=2\bigl(g(z+me_1)+g(z-me_1)-2g(z)\bigr)-2\ind_{\{i=j\}}\, .
\end{equation*}
So $\Sigma_{ii}=4\drift(4g(me_1)-2)/\log r\to4\drift/\pi$, and $\Sigma_{ij}=O(m^2/|z|^2)=O(r^{-\nf14})$ for $i\neq j$ by~\eqref{eq:kernel-asymptotics}.
\end{proof}

\begin{proof}[Proof of Theorem~\ref{thm:sharp}\,\textup{(iii)}]
By~\eqref{eq:gradient}, the increments of~$S^i$ are at most $b=Cr_n^{-\nf12}$. Lemma~\ref{lem:exp-deviation}\,(ii) with $\Omega=\Omega_n$, the brackets of Lemma~\ref{lem:separated-brackets}, $\delta=Cr_n^{-\nf14}$, and $\beta=a\sqrt{\log r_n}$ for a small $a>0$, so that $m^{-1}e^{C\beta^2}\leq r_n^{-\nf1{16}}$, gives
\begin{equation*}
\P\left(\max_iS^i_n<ca\sqrt{\log r_n}\right)\leq Cr_n^{-\nf1{16}}\, .
\end{equation*}
On~$\Omega_n$, by~\eqref{eq:pointwise} and~\eqref{eq:bulk-potential-only} with $\theta=\nf12$, the deviation $\ell_n(y)-2d\drift(r_n-|y|)$ equals $-\mathcal M_n^y+o(s_n\sqrt{\log r_n})$, uniformly over $|y|\leq r_n/2$. So, if $S^i_n\geq ca\sqrt{\log r_n}$, then for $d\geq3$ the deviation at~$y_i$ is at least $c\sqrt{r_n\log n}$, and for $d=2$ the deviations at~$y_i+me_1$ and~$y_i$ differ by at least $c\sqrt{r_n}\log n$, so one of them is at least $(c/2)\sqrt{r_n}\log n$ in absolute value. All these sites lie in $B(0,\sqrt{r_n})$, and the failure probability is at most $Cr_n^{-\nf1{16}}+O(n^{-10})\leq n^{-c}$, which proves~\eqref{eq:sharp-bulk}.
\end{proof}

\subsection{Logarithmic lower bounds for the radii}

The next proposition adapts the argument of \citet[Section~3]{AsselahGaudilliere2014}: from each of many sites visited near the sphere of radius~$r_n$, the walk leaves $B(0,r_n+h)$, for a given $h>0$, along a straight path with probability at least~$e^{-Ch}$.

\begin{proposition}[Lower bounds for the radii]\label{prop:log-lower}
Let $\gauge$ be the Euclidean norm, and let $0<\drift<\nf1d$. There exist $c(d,\drift)>0$, $C(d,\drift)<\infty$ and $h_0(d,\drift)>0$ such that the following hold.
\begin{enumerate}[label=\textup{(\roman*)}]
\item \underline{\emph{Tail bounds}}: for all sufficiently large~$n$ and every real~$h$ with $h_0\leq h\leq\nf{r_n}{8}$,
\begin{align}
\P\bigl(\Rin(n)\geq r_n-h\text{ and }\Rout(n)\leq r_n+h\bigr)&\leq\exp\bigl\{-cr_n^{d-1}e^{-Ch}\bigr\}\, ,\label{eq:log-lower-tail}\\
\P\bigl(\Rout(n)\leq r_n+h\bigr)&\leq\exp\bigl\{-cr_n^{d-1}e^{-Ch}\bigr\}+\exp\bigl\{-cr_n^{d-3}h^2\bigr\}\, .\label{eq:outer-lower-tail}
\end{align}
\item \underline{\emph{Almost-sure bounds}}: almost surely, for all sufficiently large~$n$,
\begin{equation}\label{eq:log-lower-as}
\max\{r_n-\Rin(n),\Rout(n)-r_n\}>c\log n\, ,
\end{equation}
and, if $d\geq3$, also $\Rout(n)-r_n>c\log n$.
\end{enumerate}
\end{proposition}

\begin{proof}
Let $r\coloneqq r_n$ and $\Lambda\coloneqq\{x\in\Z^d:r-2h<|x|\leq r+h\}$. Every transition probability is at least $1/(2d)-\drift/2>0$, so from each site of~$\Lambda$ the walk leaves $B(0,r+h)$ in $k\coloneqq\lceil4h\sqrt d\rceil$ outward steps along a coordinate of largest absolute value with conditional probability at least $e^{-Ch}$. Run these $k$-step trials from successive first visits to sites of~$\Lambda$ before time~$n-k$, on disjoint time intervals of length~$k$. If $|A_n\cap\Lambda|\geq c_1r^{d-1}h$, then at least $cr^{d-1}$ trials start, because each interval contains at most $k+1$ first visits; so, for every fixed $c_1>0$,
\begin{equation}\label{eq:outer-lower-trials}
\P\bigl(\Rout(n)\leq r+h\text{ and }|A_n\cap\Lambda|\geq c_1r^{d-1}h\bigr)
\leq\exp\{-cr^{d-1}e^{-Ch}\}\, .
\end{equation}
For $c_1$ small and $h_0$ large, the sites with $r-2h<|x|<r-h-1$, which are in~$A_n$ when $\Rin(n)\geq r-h$, number at least $c_1r^{d-1}h$; this proves~\eqref{eq:log-lower-tail}. If instead $\Rout(n)\leq r+h$ and $|A_n\cap\Lambda|<c_1r^{d-1}h$, then $|x|\leq\int_{C_x}|v|\dd v$ and the inclusion of the cells with $|x|\leq r-2h$ in $B(0,r-h)$ give $\sum_{x\in A_n}|x|\leq n/(2\drift)-c_2r^dh+2c_1r^dh$ for some $c_2(d)>0$ that does not depend on~$c_1$. If $c_1\leq c_2/4$, then \eqref{eq:moment-martingale} gives $\mathcal Q_n\leq-cr^dh$, and Azuma's inequality for~$\mathcal Q$ stopped on leaving $B(0,r+h)$, with increments at most~$5r$, bounds the probability of this event by $\exp\{-cr^{2d}h^2/(nr^2)\}=\exp\{-cr^{d-3}h^2\}$. With~\eqref{eq:outer-lower-trials}, this proves~\eqref{eq:outer-lower-tail}. Taking $h=c\log n$ with $c$ small makes the first tail summable in every dimension and the second for $d\geq3$, and the Borel--Cantelli lemma proves~(ii).
\end{proof}

\bibliographystyle{plainnat}
\bibliography{centrally-excited-walks}
\end{document}